\documentclass[10pt,reqno]{amsart}

\usepackage{amsfonts,amssymb,amsmath,mathrsfs}
\usepackage{float}
\usepackage{stmaryrd}
\usepackage{srcltx}
\usepackage{epsfig}
\usepackage{graphicx}
\usepackage{caption}
\usepackage{color}
\usepackage{verbatim}
\usepackage{kotex}
\usepackage{hyperref}

\title[ ]{Multi-agent system for target tracking on a sphere and its asymptotic behavior}

\author[Sun-Ho Choi]{Sun-Ho Choi}
\address[Sun-Ho Choi]{Department of Applied Mathematics and the Institute of Natural Sciences, Kyung Hee University, 1732 Deogyeong-daero, Giheung-gu, Yongin 17104, Republic of Korea}
\email{sunhochoi@khu.ac.kr}

\author[Dohyun Kwon]{Dohyun Kwon}
\address[Dohyun Kwon]{Department of Mathematics, University of Wisconsin-Madison, 480 Lincoln Dr., Madison, WI 53706, USA}
\email{dkwon7@wisc.edu}

\author[Hyowon Seo]{Hyowon Seo}
\address[Hyowon Seo]{Department of Applied Mathematics and the Institute of Natural Sciences, Kyung Hee University, 1732 Deogyeong-daero, Giheung-gu, Yongin 17104, Republic of Korea}
\email{hyowseo@gmail.com}

\begin{document}

\newtheorem{theorem}{Theorem} [section]
\newtheorem{maintheorem}{Theorem}
\newtheorem{lemma}[theorem]{Lemma}
\newtheorem{proposition}[theorem]{Proposition}
\newtheorem{remark}[theorem]{Remark}
\newtheorem{example}[theorem]{Example}
\newtheorem{exercise}{Exercise}
\newtheorem{definition}{Definition}[section]
\newtheorem{corollary}[theorem]{Corollary}

%%%%%%%%%%%%%%%%%%%%%%%%%%%%%%%%%%%%%
\def\A{{\mathcal A}}
\newcommand{\noi}{\noindent}
\newcommand{\Z}{\mathbb{Z}}
\newcommand{\R}{\mathbb{R}}
\newcommand{\C}{\mathbb{C}}
\newcommand{\T}{\mathbb{T}}
\newcommand{\bul}{\bullet}
\newcommand{\E}{\mathcal{E}}
\newcommand{\N}{\mathcal{N}}
\newcommand{\RR}{\mathcal{R}}
\newcommand{\D}{\mathcal{D}}
\newcommand{\HH}{\mathcal{H}}

\newcommand{\al}{\alpha}
\newcommand{\dl}{\delta}
\newcommand{\Dl}{\Delta}
\newcommand{\eps}{\varepsilon}
\newcommand{\kk}{\kappa}
\newcommand{\g}{\gamma}
\newcommand{\G}{\Gamma}
\newcommand{\ld}{\lambda}
\newcommand{\lam}{\lambda}
\newcommand{\Ld}{\Lambda}
\newcommand{\s}{\sigma}
\newcommand{\ft}{\widehat}
\newcommand{\wt}{\widetilde}
\newcommand{\cj}{\overline}
\newcommand{\dx}{\partial_x}
\newcommand{\dt}{\partial_t}
\newcommand{\dd}{\partial}
\newcommand{\invft}[1]{\overset{\vee}{#1}}
\newcommand{\lrarrow}{\leftrightarrow}
\newcommand{\embeds}{\hookrightarrow}
\newcommand{\LRA}{\Longrightarrow}
\newcommand{\LLA}{\Longleftarrow}

\newcommand{\wto}{\rightharpoonup}

%Japanese Bracket
\newcommand{\jb}[1]
{\langle #1 \rangle}

%color
%\newcommand{\dk}[1]{{\color{blue} #1  }}
%\definecolor{caribbean green}{rgb}{0.0, 0.8, 0.6}
\newcommand{\dk}[1]{{\color{blue}[#1]}}

%%%%%%%%%%%%%%%%%%%%%%%%%%%%%%%%%%
\newcommand{\e}{\varepsilon} %
\newcommand{\rotji}{{R_{z_j \rightarrow z_i}}}
\newcommand{\rot}{{R_{z_1 \rightarrow z_2}}}
\newcommand{\rott}{{R_{z_1 \rightarrow z_2}^T}}
\newcommand{\roti}{{R_{z_1 \rightarrow z_2}^{-1}}}
\newcommand{\rotr}{{R_{z_2 \rightarrow z_1}}}

\newcommand{\vii}{{\{v_i\}_{\idn}}}
\newcommand{\ro}{{R_{\cdot \rightarrow \cdot}}}
\newcommand{\rji}{{R_{x_j \rightarrow x_i}}}
\newcommand{\rij}{{R_{x_i \rightarrow x_j}}}
\newcommand{\rjih}{{R_{x_j \rightarrow \frac{x_i+x_j}{|x_i+x_i|}}}}
\newcommand{\rijh}{{R_{x_i \rightarrow \frac{x_i+x_j}{|x_j+x_i|}}}}

\newcommand{\hR}{{P}}
\newcommand{\hro}{{\hR_{\cdot \rightarrow \cdot}}}
\newcommand{\hrji}{{\hR_{x_j \rightarrow x_i}}}
\newcommand{\hrij}{{\hR_{x_i \rightarrow x_j}}}
\newcommand{\hrki}{{\hR_{x_k \rightarrow x_i}}}
\newcommand{\zc}{{z_1\times z_2}}
\newcommand{\xvi}{{\{(x_i, v_i )\}_{\idn}}}
\newcommand{\idn}{{1 \leq i \leq N}}
\newcommand{\xii}{{\{x_i\}_{\idn}}}
\renewcommand{\theequation}{\thesection.\arabic{equation}}
\renewcommand{\thetheorem}{\thesection.\arabic{theorem}}
\renewcommand{\thelemma}{\thesection.\arabic{lemma}}
\newcommand{\bbr}{\mathbb R}
\newcommand{\bbz}{\mathbb Z}
\newcommand{\bbn}{\mathbb N}
\newcommand{\bbs}{\mathbb S}
\newcommand{\bbp}{\mathbb P}
\newcommand{\ddiv}{\textrm{div}}
\newcommand{\bn}{\bf n}
\newcommand{\rr}[1]{\rho_{{#1}}}
\newcommand{\thh}{\theta}
\def\charf {\mbox{{\text 1}\kern-.24em {\text l}}}
\renewcommand{\arraystretch}{1.5}

\thanks{
%\textbf{Acknowledgment.}
}

\begin{abstract}
We propose a second-order multi-agent system for target tracking on a sphere. The model contains a centripetal force, a bonding force, a velocity alignment operator to the target, and cooperative control between flocking agents. We propose an appropriate regularized rotation operator instead of Rodrigues' rotation operator to derive the velocity alignment operator for target tracking. By the regularized rotation operator, we can decompose the phase of agents into translational and structural parts. By analyzing the translational part of this reference frame decomposition, we can obtain rendezvous results to the given target. If the multi-agent system can obtain the target's position, velocity, and acceleration vectors, then the complete rendezvous occurs. Even in the absence of the target's acceleration information, if the coefficients are sufficiently large enough, then the practical rendezvous occurs.\end{abstract}

\maketitle

%\tableofcontents

%%%%%%%%%%%%%%%%%%%%%%%%%%%%%%%%%%%%%%%%%%%%%%%%%%%%%%%%%%%%%%%%%%%%%%%%%%%%%%%%%%%%%%%%%%%%%%%%%%%%%%%
%
%  Section
%
%%%%%%%%%%%%%%%%%%%%%%%%%%%%%%%%%%%%%%%%%%%%%%%%%%%%%%%%%%%%%%%%%%%%%%%%%%%%%%%%%%%%%%%%%%%%%%%%%%%%%%
\section{Introduction}\label{sec1}
\setcounter{equation}{0}
Target tracking refers to designing a dynamical system that agents follow given maneuvering target agents using the information of the targets, such as position, velocity, and acceleration. The target tracking problem is applied in various fields, such as mobile sensor networks, virtual reality, and surveillance systems using unmanned aerial vehicles (UAVs) \cite{S-B, X-Z-S, Y-Y-J}.  Most of the relevant literature focuses on the uncertainty of target motions. From a technical point of view, we can divide the models for this field into measurement models, target motion models, and filtering models.  The measurement model deals with target information in a sensor coordinate containing additive noise such as image sensors and radar sensor networks \cite{B2,B1,S-S-D}. The target motion model is a coupled dynamical system for target tracking. The filtering model is based on the particle filter method and stochastic frameworks estimating the target state such as nonlinear filtering \cite{H-H,L-J} and adaptive filtering \cite{L-Z-C-Z-Z,M-C}.

%The filtering model and measurement model are closely related to each other.
%In recent years, the hybrid models have been dealt with...

Depending on the structure of the system, we also divide the models for target tracking into two types of systems:  single integrator model and double integrator model. For the single integrator model, one can control the velocity of the agents directly. For example, in \cite{D-S-A-Y}, the authors proposed a tracking algorithm for a slowly moving target using the target's position and bearing angle.  Many researchers assume agents can obtain only the target's position and bearing angle for targets maneuvering underwater. From the engineering point of view, it is a reasonable assumption. For the double integrator model, one can control the acceleration of agents. After Olfati-Saber’s seminal work \cite{Olfati}, researches for the dynamic tracking system using the double integrator model have been extensively conducted. For this kind of model,  the tracking agents can have the position and velocity information of the target. Moreover, to avoid collisions between agents or make a formation flight of the agents, a flocking algorithm and cooperative control are frequently used.

The domain or manifolds of agents are also one of the main topics in this field \cite{B-C-B-C,S-B} such as the surveillance system for the restricted area or target tracking system on the whole planet. Our goal is to provide a robust navigational feedback system for the target tracking problem on a sphere. Let $\gamma$-agent be a given target governed by the following system:
\begin{align}\begin{aligned}\label{target}
\dot q_\gamma&=p_\gamma,\\
\dot p_\gamma&=-\frac{\|p_\gamma\|^2}{\|q_\gamma\|^2} q_\gamma + U_\gamma(t),
\end{aligned}\end{align}
where  $q_\gamma\in \bbs^2$, $p_\gamma\in T_{q_\gamma}\bbs^2$, and  $U_\gamma$ are  the position, velocity, and control law   of the target agent ($\gamma$-agent) on sphere, respectively. To conserve the modulus of $q_\gamma(t)\in \bbs^2$, we additionally assume that the following condition holds for all $t\geq 0$.
\[q_\gamma(t) \perp U_\gamma(t).\]
Therefore, the control law $U_\gamma(t)$ has the following form: for some $u_\gamma(t)\in \bbr^3$,
\[U_\gamma(t)=\|q_\gamma(t)\|^2 u_\gamma(t)-\langle u_\gamma(t),q_\gamma(t)\rangle q_\gamma(t).\]
For simplicity, we assume that $u_\gamma(t)$ is continuous.

For a given $\gamma$-agent, we propose a novel multi-agent system for the target tracking on a spherical space:
\begin{align}
\begin{aligned}\label{main1}
\dot q_i(t)&=p_i(t),\\
\dot p_i(t)&=-\frac{\|  p_i\|^2}{\|q_i\|^2}  q_i + \sum_{j=1}^N\frac{\sigma_{ij}}{N}(\|  q_i\|^2   q_j - \langle   q_i,  q_j \rangle  q_i)\\
&\quad +c_q(\|  q_i\|^2   q_\gamma - \langle   q_i,  q_\gamma \rangle   q_i)+c_p(\hR_{q_\gamma \rightarrow q_i}(p_\gamma) -  p_i)+U_i,\end{aligned}
\end{align}
where $q_i\in \bbs^2$ and $p_i\in T_{q_i}\bbs^2$ are the position and  velocity of the $i$th agent, respectively. The first term on the right-hand side of the second equation is the centripetal force term to conserve the modulus of $q_i$. The second term $$\sum_{j=1}^N\frac{\sigma_{ij}}{N}(\|  q_i\|^2   q_j - \langle   q_i,  q_j \rangle  q_i)$$ is the cooperative control term between agents where the inter-particle force parameter is given by
\begin{align*}
\sigma_{ij}=\sigma(\|x_i-x_j\|^2).
\end{align*}
The next two terms, $c_q(\|  q_i\|^2   q_\gamma - \langle   q_i,  q_\gamma \rangle   q_i)$ and $c_p(\hR_{q_\gamma \rightarrow q_i}(p_\gamma) -  p_i)$, are the bonding force and a velocity alignment term between the target and the $i$th agent, respectively, where $c_q>0$ and $c_p>0$ are target tracking coefficients for the position and velocity, respectively. The last term $U_i$ is an extra control law based on the target's information, which  will be determined later in \eqref{U1} and \eqref{U2} for each purpose.

%
%Here, $-\|  p_i\|^2  q_i$ is the centripetal force term to conserve the modulus of $q_i$ and $\sum_{j=1}^N\sigma_{ij}(\|  q_i\|^2   q_j - \langle   q_i,  q_j \rangle  q_i)/N$ is the cooperative control term between agents. $c_q(\|  q_i\|^2   q_\gamma - \langle   q_i,  q_\gamma \rangle   q_i)$ and $c_p(\hR_{q_\gamma \rightarrow q_i}(p_\gamma) -  p_i)$ are the bonding force term and velocity alignment term between the target and the $i$th agent, respectively.  $c_q, c_p>0$ are target tracking coefficients for the position and velocity, respectively. $U_i$ is an extra control law based on the target's information, which  will be determined later for each purpose. See \eqref{U1} and \eqref{U2}. The  inter-particle force parameter is given by
%\begin{align*}%\label{sigma}
%\sigma_{ij}=\sigma(\|x_i-x_j\|^2).
%\end{align*}

Throughout this paper, we assume the initial data satisfies the following admissible conditions on $\mathbb{S}^2$:
  \begin{align}\label{initial}\|q_i(0)\|= 1,\quad \langle p_i(0),q_i(0)\rangle=0,\quad \mbox{ for all $i\in \{1,\ldots,N\}$}.\end{align}

\begin{definition}For a given target $(q_\gamma(t),p_\gamma(t))$, let  $\{(q_i(t),p_i(t))\}_{i=1}^N $ be the solution  to \eqref{main1}. We define the two kinds of rendezvouses.
\begin{enumerate}
  \item  An asymptotic complete rendezvous occurs between the agents and the given target, if
\[\lim_{t\to\infty }\max_{1\leq i\leq N}\|q_i(t)-q_\gamma(t)\|=0.\]
  \item  An asymptotic practical rendezvous occurs between the agents and the given target, if
\[\lim_{c_q,c_p\to \infty}\lim_{t\to\infty }\max_{1\leq i\leq N}\|q_i(t)-q_\gamma(t)\|=0.\]
\end{enumerate}
\end{definition}
In what follows, we will show that our model contains  many robust properties, including the complete rendezvous.  Even in the absence of the target acceleration information, the practical rendezvous occurs when the coefficients are large enough.   In particular, we obtain a sharp estimate of the distance between the target and agents. There are many other papers on the dynamics on $\mathbb{S}^2$ as well as $\mathbb{R}^n$, but our asymptotic analysis including exponential convergence and practical rendezvous is new on the target tracking problem, to the best of our knowledge.

%Although this model is inspired by [11], the target tracking estimate and practical rendezvous are novel topics.
%Also,
%but the results of target tracking on $\mathbb{S}^2$ are not reported  to our best knowledge.
%We emphasize that
% on the target tracking problem and spherical dynamics.

The derivation of our model is motivated by the decomposition property of flocking dynamics on a flat space. On a flat space, from  momentum conservation, the dynamics is represented by the composition of frame reference dynamics and local alignment dynamics as in \cite{Olfati}. In contrast to previous results in $\mathbb{R}^n$, it is hard to expect such a decomposition for the flocking model on $\mathbb{S}^2$. See Sections \ref{sec2} and \ref{sec3} for details. In particular, in our previous papers \cite{C-K-S,C-K-S1,C-K-S2}, we used Rodrigues' rotation operator $\ro$ to derive a flocking system on a sphere since Rodrigues' rotation operator $\ro$ is the most natural flocking operator. However, its composition is complex so that it is difficult to analyze. Moreover, it contains an unavoidable singularity at antipodal points due to its geometric characteristics. From this singularity, even though agents are located on $\mathbb{S}^2$, the vanishing point on the communication rate is necessary \cite{C-K-S}. Due to this difficulty, the target tracking problem on $\mathbb{S}^2$ has not been well understood.

%not been reported, to the authors' best knowledge.
% and it has a variety of good properties (see \cite{C-K-S} for instance). However
%On the other hand, by the geometric property of sphere,

We remove the singular term from the natural rotation operator $\ro$ to obtain \textit{a rotation operator in two dimensions}:
\begin{align}
\label{eqn:hrot}
\hR_{z_1 \rightarrow z_2} := \langle z_1, z_2\rangle  I + z_2 z_1^T - z_1 z_2^T,\quad  \hbox{ for $z_1$ and $z_2$ in a unit sphere. }
\end{align}
See also Appendix A for the motivation of the non-singularity rotation  operator $P$ and its properties. We will prove that its dynamics consists of the composition of the rigid motion part on $\mathbb{S}^2$ and the local alignment part. Using this property, we derive an $\mathbb{S}^2$-version of the reference frame decomposition in Proposition  \ref{prop3.2} and provide a sufficient condition to obtain a target tracking estimate between multiple agents $\{(q_i(t),p_i(t))\}_{i=1}^N $ and the given target $(q_\gamma(t),p_\gamma(t)) $. Moreover, by the regularity of the operator $\hR$, we can obtain the following global existence result.
% and non-vanishing communication rate or target tracking rate $c_p>0$.
 %From these properties, we can derive a reference frame decomposition in Proposition  \ref{prop3.2}.

%In general, in , it is impossible to reliably track a target without a velocity alignment operator.
%Similarly,

%We notice that  the simplified rotation operator $\hR$ in \eqref{eqn:hrot} is used to derive a novel flocking model for the target tracking. We also investigate its robust dynamical properties.
%On the other hand, the proposed novel flocking model \eqref{main1} uses the two dimensional rotation operator $P$ for the flocking operator and

\begin{maintheorem}
\label{thm0} Assume that for a  continuous function $u_\gamma$, a given target $(q_\gamma(t),p_\gamma(t))$   satisfies \eqref{target}. If the initial data  $\{(q_i(0),p_i(0))\}_{i=1}^N $ satisfies \eqref{initial} and $U_i$  is Lipschitz continuous with respect to $\{(q_i,p_i)\}_{i=1}^N$ with $\langle U_i,q_i\rangle=0$, then there exists a unique global-in-time solution $\{(q_i(t),p_i(t))\}_{i=1}^N$ to system \eqref{main1} and  $\{q_i(t)\}_{i=1}^N$ are located on $\mathbb{S}^2$ for all time $t>0$.
\end{maintheorem}

As in $\mathbb{R}^d$, we notice that the velocity alignment operator between the target and the agents plays an important role in target tracking. In particular, the bonding force between the target and the agents, $c_q(\|  q_i\|^2   q_\gamma - \langle   q_i,  q_\gamma \rangle   q_i)$, alone is not enough to track a target on $\mathbb{S}^2$. The velocity alignment operator $c_p(\hR_{q_\gamma \rightarrow q_i}(p_\gamma) -  p_i)$ is crucial for the target tracking algorithm. See the simulations in Section \ref{sec5}. In the next two theorems, we present a quantitative analysis of the velocity alignment operator with two different $U_i$'s;
\begin{align}\label{U1}
U_i=2\langle w_\gamma,  q_i\rangle (  q_i\times   p_i )+\dot w_\gamma(t)\times   q_i\end{align}
or
\begin{align}\label{U2}
 U_i=0,
\end{align}
where $w_\gamma$ is the angular velocity of the target given by
\begin{align}\label{wgamma}
w_\gamma=q_\gamma\times p_\gamma.
\end{align}

From Theorem \ref{thm1}, if the agents can obtain the exact target information containing acceleration, then the agents can accurately track the target, and the position differences between the target and the agents decay exponentially fast.
\begin{maintheorem}\label{thm1}
 Let   $(q_\gamma(t),p_\gamma(t))$ be a given target   satisfying \eqref{target} with a continuous target control $u_\gamma$ and $\{q_i(t),p_i(t)\}_{i=1}^N $ be the solution  to \eqref{main1}  satisfying  \eqref{initial}. We assume that $\sigma_{ij}= \sigma$ is a positive constant and
 \[U_i=2\langle w_\gamma,  q_i\rangle (  q_i\times   p_i )+\dot w_\gamma(t)\times   q_i,\]
where $w_\gamma$ is the angular velocity   defined in \eqref{wgamma}.

If $c_q>\sigma>0$ or
\begin{align*}
&\frac{1}{N}\sum_{i=1}^N\|p_i(0)-w_\gamma(0)\times q_i(0)\|^2\\
&\quad+\frac{\sigma}{2N^2}\sum_{i,j=1}^N   \|q_i(0)-q_j(0)\|^2+\frac{c_q}{N}\sum_{i=1}^N\|q_\gamma(0)-q_i(0)\|^2
<\sigma\left(1+\frac{c_q}{\sigma}\right)^2,
\end{align*}
  then the asymptotic complete rendezvous occurs and its convergence rate is exponential, i.e.,  there are positive constants $\mathcal{C}$, $\mathcal{D}$ such that
\[\|q_i(t)-q_\gamma(t)\|,~\|p_i(t)-p_\gamma(t)\|\leq \mathcal{C} e^{-\mathcal{D} t}.\]
\end{maintheorem}

\begin{remark}\begin{enumerate}
                \item If the above sufficient condition in Theorem \ref{thm1} does not hold, then we can find a steady-state solution. This means that the  sufficient condition is almost optimal to lead the convergence result in Theorem \ref{thm1}. See Section \ref{sec5}.
                \item
                The author in \cite{Olfati} does not deal with the estimate of the distance between the target and agents. Our model is inspired by \cite{Olfati}, but the target tracking estimate and practical rendezvous are novel.  %Also, there are many other papers on the dynamics on $\mathbb{S}^2$, but the results of target tracking on $\mathbb{S}^2$ are not reported  to our best knowledge.

                \item The derivation of $U_i$ in the above theorem is technical, but from the frame decomposition in Proposition 3.2, it is a very natural choice to obtain the complete rendezvous.
              \end{enumerate}

\end{remark}

The former one in \eqref{U1} corresponds to the case with the target acceleration, while it is unknown in the latter case \eqref{U2}. These choices with the different amounts of the target information induce the different accuracies of the target tracking. Since the target information obtained by the agents through observation is usually incomplete,  there have been many studies to overcome this incompleteness. For example, many researchers proposed target tracking systems including restricted target information \cite{D-S-A-Y,S-D-A}, communication-induced delays \cite{H,O-S-S}, and additive noise from measurement \cite{D-B,Y-S-C}. The result in Theorem \ref{thm2} below means that the large coefficients of the system allow the agents to get close enough to the target as needed without acceleration information of the target. In other words, the practical rendezvous occurs.

\begin{maintheorem}\label{thm2}

For  $(q_\gamma(t),p_\gamma(t))$  satisfying \eqref{target} with a continuous target control $u_\gamma$, let  $\{q_i(t),p_i(t)\}_{i=1}^N $ be the solution to \eqref{main1}  subject to the initial data satisfying  \eqref{initial} and
\[U_i=0.\]
Assume that $\sigma_{ij}= \sigma$ is a positive constant and  the angular velocity of the target and its time derivative are bounded\[\|w_\gamma\|,\|\dot w_\gamma\|<C_\gamma.\]
If $\|p_i(0)-p_\gamma(0)\|\ne 2$ for all $i\in \{1,\ldots,N\}$, then the asymptotic practical rendezvous occurs and
\[\|q_i(t)-q_\gamma(t)\|\leq
 \mathcal{C}e^{-\frac{\mathcal{D}}{4} t}+\frac{ \mathcal{C}}{\mathcal{D}}
,\]
where $\mathcal{C}$ is a positive constant depending on the initial data, $\sigma$, and $C_\gamma$.  The constant $\mathcal{D}$ is given by
\begin{align*}
  \mathcal{D}:=\left\{ \begin{array}{ll}
               \displaystyle ~c_p-\sqrt{-4c_q+c_p^2},&\quad \mbox{if}\quad c_p^2\geq -4 c_q,   \\
               \displaystyle   ~c_p,&\quad \mbox{if}\quad c_p^2< -4 c_q.
              \end{array}\right.
\end{align*}

\end{maintheorem}

There are technical issues in the proofs of Theorems \ref{thm1} and \ref{thm2}. We can obtain the complete rendezvous result in Theorem \ref{thm1} through Lasalle’s invariance principle with an energy functional. However, Lasalle’s invariance principle does not give a convergence rate.  An appropriate Lyapunov functional will be used to obtain the exponential convergence result. In particular, in this case, we derive a closed differential inequality by using six functionals including information on the distance between the target and agents and the distance between agents. The practical rendezvous in Theorem \ref{thm2} has a more subtle issue. It is necessary to control the distance between the target and agents through the size of the coefficients. However, it is impossible if the coefficients appear in the nonlinear higher-order terms except for the linear terms.  If we use a standard functional, the coefficients necessarily occur in the nonlinear terms due to the geometrical characteristics of $\mathbb{S}^2$. This problem will be solved by using new functionals inspired by hyperbolic geometry.

The rest of the paper is organized as follows.
In Section \ref{sec2}, we present the global-in-time existence and uniqueness of the solution to \eqref{main1} and  target tracking results for $\mathbb{R}^3$.  Section \ref{sec3} is devoted to a reference frame decomposition for the main system. From this decomposition, the solution to the main system is represented by the composition of operators for the translational part and the structural part.  Next, we reduce the  system for the structural  part to a linearized system in Section \ref{sec4}. Using this, we prove  the complete and practical rendezvouses of Theorems \ref{thm1} and \ref{thm2} in Section \ref{sec5}. In Section \ref{sec6}, we verify our analytic results using numerical simulations. Section \ref{sec7} is devoted to the summary of our results.

\section{Preliminary: Global well-posedness and   Motivations}\label{sec2}
\setcounter{equation}{0}
%In this section, we demonstrate that the solution of the main system exists and is unique. Moreover, we prove that the solution stays on the spherical surface. We present the results of the target tracking in $\mathbb{R}^3$ and discuss difficulties  and issues of target tracking problems  in a spherical surface based on this flat space case.

\subsection{The global existence and uniqueness}
In this section, we provide the proof of Theorem \ref{thm0}: there is a unique global-in-time solution  to \eqref{main1} and this solution is located on the sphere when the initial data satisfies the admissible conditions in \eqref{initial}.

For the local existence and uniqueness, we use the same argument in \cite{C-K-S,C-K-S1}.  For given $C^1$ functions $q_\gamma$, $p_\gamma$, and $w_\gamma=q_\gamma\times p_\gamma$,  we consider the following system of ODEs:
\begin{align}
\begin{aligned}\label{eqn:wel11}
\dot q_i(t)&=p_i(t),\\
\dot p_i(t)&=-\frac{\|  p_i\|^2}{\|q_i\|^2}  q_i + \sum_{j=1}^N\frac{\sigma(\|x_i-x_j\|^2)}{N}(\|  q_i\|^2   q_j - \langle   q_i,  q_j \rangle  q_i)\\
&\quad +c_q(\|  q_i\|^2   q_\gamma - \langle   q_i,  q_\gamma \rangle   q_i)+c_p( \langle q_\gamma, q_i\rangle  p_\gamma -
 \langle q_i,p_\gamma\rangle q_\gamma -  p_i)+U_i.\end{aligned}
\end{align}
 Here, we will choose $U_i=2\langle w_\gamma,  q_i\rangle (  q_i\times   p_i )+\dot w_\gamma(t)\times   q_i$ for the complete rendezvous and
$U_i=0$ for the practical rendezvous.

We assume that the initial data  $\{(q_i(0),p_i(0))\}_{i=1}^N$ satisfies  the admissible condition in \eqref{initial}. Then the right-hand side of \eqref{eqn:wel11} is Lipschitz continuous with respect to $\{(q_i,p_i)\}_{i=1}^N$ in a small neighborhood of $\{(q_i(0),p_i(0))\}_{i=1}^N$ in $\R^{6N}$. By the Picard-Lindel\"{o}f Theorem, there is the maximum time interval $[0,T_M)$ in which a solution of \eqref{eqn:wel11} exists and it is unique.

We next follow the same argument in \cite{C-K-S,C-K-S1}. On the maximum time interval $[0,T_M)$,  we  take the inner product between  the second equation of \eqref{eqn:wel11} and $x_i$  to obtain that
\begin{align}\label{eq 2.7}
\langle \dot{p}_i, q_i\rangle &=-\|  p_i\|^2
-c_p \langle   p_i,q_i\rangle.
\end{align}
By \eqref{eq 2.7} and the first equation of \eqref{eqn:wel11}, we obtain that
\begin{align*}
\frac{d}{dt}\sum_{i=1}^N|\langle p_i,q_i\rangle|^2&=2\sum_{i=1}^N ( \langle \dot{p}_i,q_i\rangle+\langle p_i,\dot{q}_i\rangle)\langle p_i,q_i\rangle  \\
&=  2\sum_{i=1}^N(\langle \dot{p}_i, q_i\rangle +\|p_i\|^2)~\langle p_i,q_i\rangle
\\
&= -2c_p\sum_{i=1}^N |\langle p_i,q_i\rangle|^2
.
\end{align*}
%Let $\displaystyle \psi_M^\epsilon=\max_{i,j}\sup_{t} \psi_{i,j}(t)$.
Note that the initial data satisfies $\sum_{i=1}^N |\langle v_i(0),x_i(0)\rangle|^2=0$. Therefore, the Gronwall inequality implies that
\begin{align*}\sum_{i=1}^N |\langle v_i(t),x_i(t)\rangle|\equiv0,\quad \mbox{for}~ t>0,\end{align*}
and this implies  that
\begin{align*}\langle v_i(t),x_i(t)\rangle\equiv 0.\end{align*}
We take the inner product between $\dot{q}_i$ and $q_i$.  By the first equation of \eqref{eqn:wel11},
\begin{align*}
\frac{d}{dt}\| q_i\|^2=2\langle \dot{q}_i,q_i\rangle =2\langle p_i,q_i\rangle=0.
\end{align*}
Since  initial conditions satisfy $\|x_i(0)\|=1$ and $\langle v_i(0),x_i(0)\rangle=0$ for all $i\in \{1,\ldots,N\}$, we have
\begin{align*}\|x_i(t)\|\equiv 1, \quad \mbox{for}~t>0,~ i\in\{1,\ldots N\} . \end{align*}

In conclusion, we can apply the extensibility of solutions in \cite[Corollary 2.2]{Tes12} to obtain that\[T_{M} =  \infty.\] Moreover, we can easily check that $\{(q_i(t),p_i(t))\}_{i=1}^N $ is the unique solution to \eqref{main1} by a standard argument.  Therefore, we can obtain the following proposition.

\begin{proposition}%\label{prop 3.1}
Let $\{(q_i(t),p_i(t))\}_{i=1}^N $ be a solution to \eqref{main1} with \eqref{initial}. Then  for all $i\in \{1,\ldots,N\}$ and  $t>0$,
\begin{align*}\langle q_i(t),p_i(t)\rangle=0 \quad\hbox{ and } \quad \|q_i(t)\|= 1.\end{align*}
\end{proposition}

\subsection{Target tracking problem in $\mathbb{R}^3$}\label{sec 2.3}
%
%
%In this part, we briefly review the flocking model and target tracking problem in $\bbr^3$. Many researchers have been interested in the flocking phenomenon and target tracking for multiple agents. Among them, the notable result is the following model in Olfati-Saber's seminal paper \cite{Olfati}.
%\begin{align*}
%\dot q_i&=p_i,\\
%\dot p_i&=\sum_{j=1}^N \frac{\psi_{ij}}{N}(p_j-p_i)+\sum_{j=1}^N\frac{\sigma_{ij}}{N}(q_j-q_i) +c_q(q_{\gamma}-q_i)+c_p(p_{\gamma}-p_i)+u_i,
%\end{align*}
%where $q_i\in \mathbb{R}^3$ and $p_i\in \mathbb{R}^3$ are the position and velocity of the $i$th agent, respectively and $u_i$ is a additional input parameter to be determined later. $q_\gamma,p_\gamma,u_\gamma$ are  given position, velocity, acceleration of a given target or $\gamma$-agent satisfying
%\begin{align*}
%  \dot q_\gamma&=p_\gamma,\\
%   \dot p_\gamma&=u_\gamma.
%\end{align*}
%
%
%We note that for the case of target tracking problem, the effect of the flocking term

In this section, we estimate the distance between the target and agents for the following model in $\mathbb{R}^3$:
\begin{align*}
\dot q_i&=p_i,\\
\dot p_i&=\sum_{j=1}^N \frac{\psi_{ij}}{N}(p_j-p_i)+\sum_{j=1}^N\frac{\sigma_{ij}}{N}(q_j-q_i) +c_q(q_{\gamma}-q_i)+c_p(p_{\gamma}-p_i) + u_i,
\end{align*}
where $q_i\in \mathbb{R}^3$ and $p_i\in \mathbb{R}^3$ are the position and velocity of the $i$th agent, respectively. Here, $q_\gamma, p_\gamma$, and $u_\gamma$ are the position, velocity, and acceleration of a given target ($\gamma$-agent) satisfying
\begin{align*}
\dot q_\gamma&=p_\gamma,\\
\dot p_\gamma&=u_\gamma.
\end{align*}
A new input parameter $u_i$ will be determined later. Depending on the information of the target, we choose two different $u_i$'s and analyze the corresponding asymptotic behaviors. The argument is straightforward, and thus the reader familiar with target tracking problems in $\mathbb{R}^3$ may skip this section.

%This robust analysis in $\mathbb{R}^3$ motivates our model \eqref{main1} on $\bbs^2$.

If $u_i = 0$, then the above model corresponds to the one in Olfati-Saber's seminal paper \cite{Olfati}. As studied in \cite{Olfati}, the system of equations can be decomposed as two second-order systems for the structural dynamics and translational dynamics. For simplicity, we assume that $\psi_{ij}=0$ and $\sigma_{ij}=\sigma_{ji}$ for all indices $i$ and $j$ in $\{1,\ldots,N\}$. We note that the effect of the flocking term $ \sum_{j=1}^N \frac{\psi_{ij}}{N}(p_j-p_i)$ is negligible, when $\max_{1\leq i,j\leq N}\|p_j-p_i\|\ll 1$. See the numerical simulations in Figures \ref{fig5} and \ref{fig6}.

Let
\[q_c=\frac{1}{N}\sum_{i=1}^N q_i,\quad p_c=\frac{1}{N}\sum_{i=1}^N p_i,\]
and
\begin{align}\label{eq 2.25}
x_i=q_i-q_c,\quad v_i=p_i-p_c.
\end{align}
Then, the above dynamics can be decomposed into the translational dynamics \eqref{eq 2.105} and the structural dynamics \eqref{eq 2.106}:
\begin{align}\begin{aligned}\label{eq 2.105}
\dot q_c&=p_c,\\
\dot p_c&= c_q(q_{\gamma}-q_c)+c_p(p_{\gamma}-p_c)+u_i,
\end{aligned}
\end{align}
and
\begin{align}\begin{aligned}\label{eq 2.106}
\dot x_i&=x_i,\\
\dot v_i&=\sum_{j=1}^N\frac{\sigma_{ij}}{N}(x_j-x_i) -c_q x_i-c_p v_i.
\end{aligned}\end{align}The structural dynamics part in \eqref{eq 2.106} has been analyzed in \cite{Olfati}.

We focus on the translational dynamics part in \eqref{eq 2.105} for two different cases of $u_i$. We first suppose that all of the position $p_\gamma$, velocity $q_\gamma$, and acceleration $u_\gamma$ of the target are given. In this case, it is natural to choose $u_i:=u_\gamma$. Let
\[q_d=q_c-q_{\gamma},\quad p_d=p_c-p_{\gamma}.\]
Then the translational dynamics in \eqref{eq 2.105} can be rewritten as
\begin{align*}
\dot q_d&=p_d,\\
\dot p_d&= -c_q q_d-c_p p_d.
\end{align*}
This is a simple linear system of ODEs  and it has the following solution;
\begin{align*}
q_d(t)&=\frac{1}{2 \sqrt{c_p^2-4 c_q}}
\bigg[-c_p q_d(0) e^{\frac{1}{2} t \left(-\sqrt{c_p^2-4 c_q}-c_p\right)}+q_d(0) \sqrt{c_p^2-4 c_q} e^{\frac{1}{2} t \left(-\sqrt{c_p^2-4 c_q}-c_p\right)}
\\&\quad\qquad \qquad \qquad\qquad  +c_p q_d(0) e^{\frac{1}{2} t \left(\sqrt{c_p^2-4 c_q}-c_p\right)}+q_d(0) \sqrt{c_p^2-4 c_q} e^{\frac{1}{2} t \left(\sqrt{c_p^2-4 c_q}-c_p\right)}\\
&\quad\qquad\qquad \qquad \qquad  \qquad -2 p_d(0) e^{\frac{1}{2} t \left(-\sqrt{c_p^2-4 c_q}-c_p\right)}+2 p_d(0) e^{\frac{1}{2} t \left(\sqrt{c_p^2-4 c_q}-c_p\right)}\bigg].
\end{align*}
Therefore, we can easily check that $q_d$ and $p_d$ converge  to zero exponentially. This means that the complete rendezvous with an exponential decay rate occurs for any positive $c_q$ and $c_p$.

If we only know the position and velocity of the target, we cannot expect a complete rendezvous. On the other hand, we can control the maximum position difference between the target and agents if the tracking coefficients for the target are sufficiently large. We refer to \cite{C-C-H,C-C-H2} for related issues.

%Next, we consider the practical rendezvous.
For $u_i=0$, the translational dynamics is given by
\begin{align*}
\dot q_d&=p_d,\\
\dot p_d&= -c_q q_d-c_p p_d-u_\gamma.
\end{align*}
As we mentioned above, we cannot expect the complete rendezvous for this case. Alternatively, to obtain the practical rendezvous estimate, we additionally assume that the acceleration of the target is bounded:
 \begin{align}\label{assump_gamma}
 \limsup \|u_\gamma\|\leq C_\gamma,
 \end{align}
 for some $C_\gamma>0$. Then we define auxiliary variables as follows.
\[X^1_d=\langle q_d,q_d\rangle,\quad X^2_d=\langle q_d,p_d\rangle,\quad X^3_d=\langle p_d,p_d\rangle. \]

By the system of the translational dynamics, we can obtain
\begin{align*}
\dot X^1_d&=2 X^2_d,\\
\dot X^2_d&=X^3_d -c_qX^1_d-c_pX^2_d-\langle q_d, u_\gamma \rangle,\\
\dot X^3_d&= -2c_qX^2_d-2c_pX^3_d-2\langle p_d, u_\gamma \rangle.
\end{align*}
We rewrite the above system of equations as the following  inhomogeneous linear system of ODEs:
\begin{align*}
\dot X_d=A_d X_d+F_d,
\end{align*}
where  $X_d=(X^1_d,X^2_d,X^3_d)^T$ and  $F_d=(0,-\langle q_d, u_\gamma \rangle,-2\langle p_d, u_\gamma \rangle)^T$, and  the coefficient matrix is given by
\begin{align*}
M_d=
\begin{bmatrix}
0 &2 &0  \\
 -c_q & -c_p  &1 \\
0 &-2c_q & -2c_p
\end{bmatrix}.
\end{align*}
Note that $M_d$ has the following eigenvalues.
\[\left\{- c_p , ~- c_p -\sqrt{c_p^2-4c_q},~- c_p +\sqrt{c_p^2-4c_q}\right\}.\]

Let $D_d<0$ be the greatest real part in the above eigenvalues and let
\[-\mu_d=D_d.\]
Then, we have
\begin{align*}
\frac{d}{dt}\|X_d\|^2&= 2\langle X_d,M_d X_d\rangle +2\langle X_d,F_d\rangle\\&\leq -2\mu_d  \|X_d\|^2+ 2\|X_d\|\|F_d\|,
\end{align*}
this implies that
\begin{align*}
\frac{d}{dt}\|X_d\| \leq -\mu_d  \|X_d\|+ \|F_d\|.
\end{align*}
From  elementary calculations, it follows that for any $\epsilon>0$,
\begin{align*}
\|F_d\|&\leq \|q_d\|\|u_\gamma\|+2\|p_d\|\|u_\gamma\|\\
&\leq \frac{\epsilon\|q_d\|^2}{2}+\frac{1}{2\epsilon}\|u_\gamma\|^2+\frac{\epsilon\|p_d\|^2}{2}+\frac{2}{\epsilon}\|u_\gamma\|^2\\
& \leq \epsilon\|X_d\|+\frac{5}{2\epsilon}\|u_\gamma\|^2.
\end{align*}
We choose $\displaystyle \epsilon=\mu_d/2$ and use the Gronwall inequality and \eqref{assump_gamma} to obtain  that
\begin{align*}
\|X_d\|&\leq  e^{-(\mu_d-\epsilon)t}\|X_d(0)\|+\frac{5}{2\epsilon}e^{(\mu_d-\epsilon)t}\int_0^t \|u_\gamma(s)\|^2e^{(\mu_d-\epsilon)s}ds\\
&\leq  e^{-(\mu_d-\epsilon)t}\|X_d(0)\|+C_\gamma^2\frac{5}{2\epsilon}e^{-(\mu_d-\epsilon)t}
\frac{e^{(\mu_d-\epsilon)t}-1}{\mu_d-\epsilon}.
\end{align*}
This implies that
\[\limsup \|X_d\|\leq \frac{10C_\gamma^2}{\mu_d^2}.\]
Thus, if we choose a sufficiently large tracking coefficients $c_q,c_p>0$, then we obtain that
\[\limsup_{t\to \infty} \|q_i(t)-q_\gamma(t)\|,~\limsup_{t\to \infty} \|p_i(t)-p_\gamma(t)\|\ll 1.\]

\section{Generalized rotation operator on sphere and reference frame decomposition}\label{sec3}
\setcounter{equation}{0}
%In this section, we present a generalized rotation operator $S_\gamma$ along the given target and by using this operator, we obtain the reference frame decomposition for the main system \eqref{main1}. To consider the generalized rotation operator, we first consider a given $\gamma$-agent trajectory on $\mathbb{S}^2$:

In this section, we decompose our model \eqref{main1} on $\mathbb{S}^2$ into structural dynamics and translational dynamics. Due to the complexity of \eqref{main1}, the decomposition of agents' positions into a sum of two vectors as the model in $\bbr^3$ is not suitable for our case. Instead, we observe that a rigid body motion on $\mathbb{S}^2$ can be used as a reference frame. Choosing an appropriate rigid body motion, our model can be represented as the composition of a rigid body motion and local alignment dynamics. The rigid body motion can be derived based on the angular velocity tensor $W_\gamma(t)$ of the  $\gamma$-agent and a generalized rotation operator $S_\gamma$ along the given target described below. Recall the given $\gamma$-agent trajectory on $\mathbb{S}^2$:
\[\dot q_\gamma= p_\gamma,\]
where $q_\gamma\in \bbs^2$ and $p_\gamma\in T_x\bbs^2$ are the position and velocity of the given $\gamma$-agent, respectively.

 Let \[w_\gamma=q_\gamma\times p_\gamma.\]
By elementary calculation, we have $q_\gamma\times w_\gamma=-p_\gamma$ and
\begin{align*}
\dot q_\gamma=w_\gamma \times q_\gamma.
\end{align*}
For the angular velocity vector $w_\gamma=(w_\gamma^{1},w_\gamma^{2},w_\gamma^{3})^T$, we define  the angular velocity tensor $W_\gamma(t)$ of the  $\gamma$-agent  by
\begin{align*}
W_\gamma^t=
\begin{bmatrix}
0 & -w_\gamma^{3}(t) &w_\gamma^{2}(t)  \\
 w_\gamma^{3}(t) & 0  &-w_\gamma^{1}(t) \\
-w_\gamma^{2}(t) &w_\gamma^{1}(t) & 0
\end{bmatrix}.
\end{align*}
From the above notation, the equation for the $\gamma$-agent is written by
\begin{align*}
\dot q_\gamma=p_\gamma=W_\gamma^t q_\gamma.
\end{align*}

Now, we consider the following system of ODEs:
\begin{align}\label{eq W}
\dot x(t)=W_\gamma^tx(t).
\end{align}
We can define  the corresponding solution operator $S_\gamma(x_0,t)=S_\gamma^tx_0: \bbs^2\times[0,\infty) \mapsto \bbs^2$ such that
\begin{align}\label{eq W2}S_\gamma^tx_0=x(t;x_0),\end{align}
where $x(t;x_0)$  is the solution to \eqref{eq W} subject to
\begin{align}\label{eq W3}x(0;x_0)=x_0\in \bbs^2.\end{align}
One can easily check that $S_\gamma^t$ is a rigid body motion on $\bbs^2$.

\begin{lemma}\label{lemma 3.1}Let  $x_\gamma(t)\in \mathbb{S}^2$  be the position of a $\gamma$-agent which is a $C^2$ function with respect to $t\geq 0$.
For the given $\gamma$-agent, the solution operator $S_\gamma^t$ defined above is represented by a matrix and the matrix product. Moreover, for any $x,y\in \mathbb{R}^3$,
\[\|x\|^2=\|S_\gamma^tx\|^2,\quad \langle x,y \rangle=\langle S_\gamma^tx,S_\gamma^ty \rangle.\]

\end{lemma}
\begin{proof}
Let $x_\gamma(t)$ be a given $C^2$ function with $\|x_\gamma(t)\|=1$. We define the solution operator
$S_\gamma^t$ by \eqref{eq W}-\eqref{eq W3}. Take any two vectors $x_1^0$ and $x_2^0$ on $\mathbb{S}^2$. Let\phantom{\eqref{eq W2}}
\[x_1(t)=S_\gamma^tx_1^0,\quad x_2(t)=S_\gamma^tx_2^0.\]
Equivalently,
\[\dot x_1(t)=W_\gamma^t x_1(t),\quad \dot x_2(t)=W_\gamma^t x_2(t),\]
subject to
\[x_1(0)=x_1^0,\quad x_2(0)=x_2^0.\]

Then we have
\[\dot x_1(t)-\dot x_2(t)=W_\gamma^t (x_1(t)-x_2(t)).\]
This implies that

\begin{align*}
\frac{1}{2}\frac{d}{dt}\| x_1(t)- x_2(t)\|^2&=\langle x_1(t)-x_2(t), W_\gamma^t (x_1(t)-x_2(t))\rangle
.\end{align*}
We note that $W_\gamma$ is a skew symmetric matrix and this implies that
\begin{align*}\langle x_1(t)-x_2(t), W_\gamma^t (x_1(t)-x_2(t))\rangle&=
\langle W_\gamma^T (t)(x_1(t)-x_2(t)), x_1(t)-x_2(t)\rangle
\\&=-\langle W_\gamma (t)(x_1(t)-x_2(t)), x_1(t)-x_2(t)\rangle
\\&=-\langle  x_1(t)-x_2(t), W_\gamma (t)(x_1(t)-x_2(t))\rangle.
  \end{align*}
Therefore, we can obtain that
\[\langle x_1(t)-x_2(t), W_\gamma^t (x_1(t)-x_2(t))\rangle=0\]
 and
\begin{align*}
\frac{d}{dt}\| x_1(t)- x_2(t)\|^2=0.\end{align*}

Since we choose $x_1^0$ and $x_2^0$ arbitrary, $S_\gamma^t: \bbs^2 \mapsto \bbs^2$ is a rigid body motion of $\bbs^2$. This implies that $S_\gamma^t$ is represented by a matrix and the matrix product. Moreover, the following holds.
\[\|x\|^2=\|S_\gamma^tx\|^2,\quad \langle x,y \rangle=\langle S_\gamma^tx,S_\gamma^ty \rangle,\]
for any $x,y\in \mathbb{R}^3$.
\end{proof}

In $\bbr^3$, the agent's position can be decomposed into a sum of two vectors as described in \eqref{eq 2.25}-\eqref{eq 2.106}. Similarly, the agent's position on $\bbs^2$ is expressed as the composition of the translational operator $S_\gamma^t$ and the structural vector $x_i$:
\begin{align}\label{eq 3.101}
q_i(t)=S_\gamma^tx_i(t).
\end{align}
Notice that  $x_\gamma(t):=q_\gamma(0)$ is a time-independent fixed point on $\mathbb{S}^2$ and satisfies
\begin{align}
\label{eq 3.102}
q_\gamma(t)=S_\gamma^tx_\gamma(t).
\end{align}
In the proposition below, we derive a second-order system of $x_i$ in the moving frame.

%
%
%
%We note that in the flat space case, the agent's position is decomposed by the sum of two vectors as in \eqref{eq 2.25}. From this motivation, the agent's position is expressed as the composition of the translation part $S_\gamma^t$ and the structural part $x_i, x_\gamma$:
%\begin{align}\label{eq 3.101}
%q_i(t)=S_\gamma^tx_i(t),
%\end{align}
%and
%\begin{align}\label{eq 3.102}q_\gamma(t)=S_\gamma^tx_\gamma(t).\end{align}
%We notice that $x_\gamma(t)$ is a time-independent fixed point on $\mathbb{S}^2$:
%\[x_\gamma(t)=q_\gamma(0).\]

\begin{proposition}\label{prop3.2}Let $(q_\gamma(t),p_\gamma(t))$ be a given $\gamma$-agent satisfying
\[\dot q_\gamma= p_\gamma,\]
where $q_\gamma\in \bbs^2$ and $p_\gamma\in T_x\bbs^2$. Let $S_\gamma^t$ be the solution operator defined by \eqref{eq W}-\eqref{eq W3}.
If \eqref{eq 3.101} and \eqref{eq 3.102} hold, then the followings are equivalent.
\begin{enumerate}
  \item  $\{(x_i(t),v_i(t))\}_{i=1}^N$ satisfies  the following structural system of ODEs:
\begin{align}\label{eq structure}
\begin{aligned}
\dot{x}_i&=v_i,\\
\dot{v}_i&=-\frac{\|v_i\|^2}{\|x_i\|^2}x_i + \sum_{j=1}^N\frac{\sigma_{ij}}{N}(\|x_i\|^2 x_j - \langle x_i,x_j \rangle  x_i)\\
&\qquad\qquad\qquad\qquad+c_q(\|x_i\|^2 x_\gamma - \langle x_i,x_\gamma \rangle  x_i)
-c_pv_i
+ A_i,
\end{aligned}
\end{align}
subject to initial data $x_i(0)\in \bbs^2$, $v_i(0)\in T_{x_i(0)}\bbs^2$ for all $i\in\{1,\ldots,N\}$.
  \item  $\{(q_i(t),p_i(t))\}_{i=1}^N$ is the solution to main system \eqref{main1} subject to \eqref{initial} with
\begin{align}\label{eq 3.65}
U_i=2\langle w_\gamma, q_i\rangle ( q_i\times  p_i )+\dot w_\gamma(t)\times  q_i+
S_\gamma^tA_i.
\end{align}
%Here, $x_\gamma$ is a given steady point on the unit sphere.

\end{enumerate}
\end{proposition}
\begin{proof}
For any $x_0\in \bbs^2$, we consider $x(t)=S_\gamma^tx_0$. Then
 \begin{align}\label{eq 3.0}\dot S_\gamma^tx_0=\frac{d}{dt}(S_\gamma^tx_0)=\dot x(t)= W_\gamma^tx(t)= W_\gamma^tS_\gamma^tx_0.\end{align}
Since $x_0$ is arbitrary and $S_\gamma^t$ is a $3\times3$ matrix by Lemma \ref{lemma 3.1}, we have
\begin{align}\label{eq 3.2}\dot S_\gamma^t=W_\gamma^tS_\gamma^t.\end{align}
We note that for any $x\in \mathbb{R}^3$,
\begin{align}\label{eq cross}
W_\gamma^t x=w_\gamma\times x.
\end{align}

We first prove that if $\{(x_i(t),v_i(t))\}_{i=1}^N$ satisfies \eqref{eq structure}, then $\{(q_i(t),\dot q_i(t))\}_{i=1}^N$ is the solution to the main system with \eqref{eq 3.65}, where  $ q_i(t)=S_\gamma^tx_i(t)$.
By the definition,
\[\frac{d}{dt}  q_i=\dot S_\gamma^tx_i+ S_\gamma^t\dot x_i.\]
Motivated by the above, we naturally define the corresponding velocity as follows.
\begin{align}\label{eq 3.10}
  p_i=\dot S_\gamma^tx_i+ S_\gamma^t\dot x_i.
\end{align}
Thus, we have
\begin{align*}
\frac{d}{dt}  p_i=\ddot S_\gamma^tx_i+ 2\dot S_\gamma^t\dot x_i+ S_\gamma^t\ddot x_i.
\end{align*}
By \eqref{eq structure} and Lemma \ref{lemma 3.1},
\begin{align}\label{eq 3.1}
\begin{aligned}
S_\gamma^t\ddot x_i
=&-\frac{\|v_i\|^2}{\|x_i\|^2}S_\gamma^t x_i + \sum_{j=1}^N\frac{\sigma_{ij}}{N}\left[\|x_i\|^2 S_\gamma^t x_j - \langle x_i,x_j \rangle  S_\gamma^tx_i\right]\\
&\quad+c_q\left[\|x_i\|^2 S_\gamma^tx_\gamma - \langle x_i,x_\gamma \rangle S_\gamma^t x_i\right]-c_p S_\gamma^tv_i+
S_\gamma^tA_i.
\end{aligned}
\end{align}
From the property of $S_\gamma^t$ in  Lemma \ref{lemma 3.1}, it follows that
\[\|x_i\|^2=\|S_\gamma^tx_i\|^2,\quad \langle x_i,x_j \rangle=\langle S_\gamma^tx_i,S_\gamma^tx_j \rangle.\]

As \cite{C-K-S,C-K-S1,C-K-S2}, we can easily prove that
\begin{align}\label{perp}
x_i(t)\in \bbs^2,\quad v_i(t)\in T_{x_i(t)}\bbs^2,\quad \mbox{ for all}\quad t\geq0,~  i\in\{1,\ldots,N\}.\end{align}
By this modulus conservation and \eqref{eq 3.1},
\begin{align}\begin{aligned}
\label{eq s1}
S_\gamma^t\ddot x_i
=&-\|v_i\|^2  q_i + \sum_{j=1}^N\frac{\sigma_{ij}}{N}\left[\|  q_i\|^2   q_j - \langle   q_i,  q_j \rangle    q_i\right]\\
&\quad+c_q\left[\|  q_i\|^2   q_\gamma - \langle   q_i,  q_\gamma \rangle   q_i\right]+c_p\left[W_\gamma^t  q_i -  p_i\right]+
S_\gamma^tA_i.
\end{aligned}\end{align}
Here, we used \eqref{eq 3.2} and \eqref{eq 3.10} to obtain
\begin{align}\label{eq 3.3}- S_\gamma^tv_i=W_\gamma^t  q_i -  p_i.\end{align}

By \eqref{eq 3.0}, \eqref{eq 3.2}, \eqref{eq 3.3} and the definition of $  q_i$ and $  p_i$,
\begin{align}\begin{aligned}\label{eq 3.161}
\ddot S_\gamma^tx_i+ 2\dot S_\gamma^t\dot x_i&=\dot W_\gamma^tS_\gamma^tx_i +W_\gamma^t\dot S_\gamma^tx_i+ 2\dot S_\gamma^t\dot x_i\\
&=\dot W_\gamma^t  q_i +W_\gamma^tW_\gamma^t   q_i+ 2W_\gamma^t S_\gamma^t\dot x_i\\
&=\dot W_\gamma^t  q_i +W_\gamma^tW_\gamma^t   q_i+ 2W_\gamma^t (  p_i-W_\gamma^t)  q_i
)
\\
&=\dot W_\gamma^t  q_i -W_\gamma^tW_\gamma^t   q_i+ 2W_\gamma^t  p_i
.
\end{aligned}
\end{align}
Clearly, by the skew symmetric property of $W_\gamma$,
\begin{align*}
\|  p_i\|^2&=\|W_\gamma^tS_\gamma^tx_i\|^2+2\langle W_\gamma^tS_\gamma^tx_i,S_\gamma^t\dot x_i\rangle + \|S_\gamma^t\dot x_i\|^2\\
&=\|W_\gamma^t   q_i\|^2+2\langle W_\gamma^t   q_i,  p_i -W_\gamma   q_i\rangle + \|v_i\|^2
\\
&=\langle    q_i,W_\gamma^tW_\gamma^t   q_i\rangle -2\langle    q_i,W_\gamma^t  p_i\rangle + \|v_i\|^2.
\end{align*}
This implies that
\begin{align}\label{eq 3.162}
-\|v_i\|^2
=-\|  p_i\|^2+\langle    q_i,W_\gamma^tW_\gamma^t   q_i\rangle -2\langle    q_i,W_\gamma^t  p_i\rangle .
\end{align}
By \eqref{eq 3.161} and \eqref{eq 3.162}, we have
\begin{align}\begin{aligned}\label{eq s2}
\ddot S_\gamma^tx_i+ 2\dot S_\gamma^t\dot x_i-\|v_i\|^2  q_i
&=-\|  p_i\|^2  q_i
+\langle    q_i,W_\gamma^tW_\gamma^t   q_i\rangle  q_i
-W_\gamma^tW_\gamma^t   q_i
\\
&\quad-2\langle   q_i,W_\gamma^t   p_i\rangle  q_i
+ 2W_\gamma^t  p_i
+\dot W_\gamma^t  q_i
.\end{aligned}
\end{align}

Thus, by  \eqref{eq s1} and  \eqref{eq s2},
\begin{align}\begin{aligned}\label{eq 3.21}
\dot  p&=\ddot S_\gamma^tx_i+ 2\dot S_\gamma^t\dot x_i+S_\gamma^t\ddot x_i
\\
&=-\|  p_i\|^2  q_i + \sum_{j=1}^N\frac{\sigma_{ij}}{N}(\|  q_i\|^2   q_j - \langle   q_i,  q_j \rangle    q_i)+c_q(\|  q_i\|^2   q_\gamma - \langle   q_i,  q_\gamma \rangle   q_i)\\&\quad+c_p(W_\gamma^t  q_i -  p_i)+\langle    q_i,W_\gamma^tW_\gamma^t   q_i\rangle  q_i
-W_\gamma^tW_\gamma^t   q_i\\
&\quad
-2\langle   q_i,W_\gamma^t   p_i\rangle  q_i
+ 2W_\gamma^t  p_i
+\dot W_\gamma^t  q_i
+
S_\gamma^tA_i
.
\end{aligned}
\end{align}
We note that for any $x\in \mathbb{R}^3$,
\begin{align}\label{eq cross}
W_\gamma^t x=w_\gamma\times x.
\end{align}
From \eqref{eq 3.21}-\eqref{eq cross} and the modulus conservation property of $S_\gamma^t$ with $x_i(t)\in \bbs^2$, it follows that
\begin{align*}
\dot  p
&=-\frac{\|  p_i\|^2}{\|q_i\|^2}  q_i + \sum_{j=1}^N\frac{\sigma_{ij}}{N}(\|  q_i\|^2   q_j - \langle   q_i,  q_j \rangle    q_i)+c_q(\|  q_i\|^2   q_\gamma - \langle   q_i,  q_\gamma \rangle   q_i)\\
&\quad+c_p(w_\gamma\times  q_i -  p_i)+2\langle w_\gamma,  q_i\rangle (  q_i\times   p_i )+\dot w_\gamma\times   q_i+
S_\gamma^tA_i.
\end{align*}

Now, if we choose $A_i$ such as
\[2\langle w_\gamma,  q_i\rangle (  q_i\times   p_i )+\dot w_\gamma(t)\times   q_i+
S_\gamma^tA_i=0,
\]
then our model corresponds to $u_i=0$ case in the flat space case, and if we choose $A_i=0$ then our model corresponds to $u_i=u_\gamma$ case in the flat space case.
From the uniqueness of the solution to the main system, we obtain the desired result.
\\

We next prove that if $\{(q_i(t),p_i(t))\}_{i=1}^N$ is the solution to the main system  with \eqref{eq 3.65}, then $\{(x_i(t),\dot x_i(t))\}_{i=1}^N$ satisfies \eqref{eq structure}, where $ x_i(t)=S_\gamma^{-1}(t)q_i(t)$. By the first equation of \eqref{main1}, we have
\begin{align}\label{eq 3.11}
p_i=\dot q_i=\dot S_\gamma^tx_i+ S_\gamma^t\dot x_i=W_\gamma^t S_\gamma^tx_i+ S_\gamma^t\dot x_i.
\end{align}
This implies that
\begin{align}\begin{aligned}\label{eq 3.120}
\ddot q_i&=\ddot S_\gamma^tx_i+2\dot S_\gamma^t\dot x_i+ S_\gamma^t\ddot
x_i\\
&=\dot W_\gamma^tS_\gamma^tx_i+W_\gamma^tW_\gamma^tS_\gamma^tx_i+2W_\gamma^tS_\gamma^t\dot x_i+ S_\gamma^t\ddot
x_i.\end{aligned}
\end{align}
By \eqref{eq 3.11}, we have
\begin{align}\begin{aligned}
\label{eq 3.121}
\|  p_i\|^2&=\|W_\gamma^tS_\gamma^tx_i\|^2+2\langle W_\gamma^tS_\gamma^tx_i,S_\gamma^t\dot x_i\rangle + \|S_\gamma^t\dot x_i\|^2\\
&=\|W_\gamma^t   q_i\|^2+2\langle W_\gamma^t   q_i,  p_i -W_\gamma   q_i\rangle + \|v_i\|^2
\\
&=\langle    q_i,W_\gamma^tW_\gamma^t   q_i\rangle -2\langle    q_i,W_\gamma^t  p_i\rangle + \|v_i\|^2.
\end{aligned}\end{align}
The second equation in \eqref{main1} and $q_i(t)\in \bbs^2$ imply that
\begin{align*}
\begin{aligned}
\ddot q_i
&=-\|  p_i\|^2  q_i + \sum_{j=1}^N\frac{\sigma_{ij}}{N}(\|  q_i\|^2   q_j - \langle   q_i,  q_j \rangle  q_i) +c_q(\|  q_i\|^2   q_\gamma - \langle   q_i,  q_\gamma \rangle   q_i)\\
&\quad+c_p(\hR_{q_\gamma \rightarrow q_i}(p_\gamma) -  p_i)+U_i
,\\
&=-\|  p_i\|^2  S_\gamma^t x_i
 + \sum_{j=1}^N\frac{\sigma_{ij}}{N}(\|   S_\gamma^t x_i\|^2   S_\gamma^t x_j
 - \langle   S_\gamma^t x_i
,  S_\gamma^t x_j
 \rangle  S_\gamma^t x_i
)\\
&\quad +c_q(\|  S_\gamma^t x_i
\|^2    S_\gamma^t x_\gamma - \langle   S_\gamma^t x_i
,   S_\gamma^t x_\gamma\rangle   S_\gamma^t x_i
)+c_p(\hR_{q_\gamma \rightarrow q_i}(p_\gamma) -  p_i)+U_i.
\end{aligned}
\end{align*}
From the property of $S_\gamma^t$ in Lemma \ref{lemma 3.1}, it follows that
\begin{align}
\begin{aligned}\label{eq 3.119}
\ddot q_i
&=-\|  p_i\|^2  S_\gamma^t x_i
 + \sum_{j=1}^N\frac{\sigma_{ij}}{N}(\|    x_i\|^2   S_\gamma^t x_j
 - \langle   x_i
,  x_j
 \rangle  S_\gamma^t x_i
)\\
&\quad +c_q\left(\|   x_i
\|^2    S_\gamma^t x_\gamma - \langle   x_i
,    x_\gamma\rangle   S_\gamma^t x_i
\right)+c_p(\hR_{q_\gamma \rightarrow q_i}(p_\gamma) -  p_i)+U_i
.
\end{aligned}
\end{align}

By \eqref{eq 3.120}--\eqref{eq 3.119}, \phantom{\eqref{eq 3.121}}
\begin{align*}
\begin{aligned} S_\gamma^t\ddot
x_i
&=-\left[
\dot W_\gamma^tS_\gamma^tx_i+W_\gamma^tW_\gamma^tS_\gamma^tx_i+2W_\gamma^tS_\gamma^t\dot x_i\right] \\&\quad-\|  p_i\|^2  S_\gamma^t x_i+ \sum_{j=1}^N\frac{\sigma_{ij}}{N}(\|    x_i\|^2   S_\gamma^t x_j
 - \langle   x_i
,  x_j
 \rangle  S_\gamma^t x_i
)\\
&\quad +c_q\left(\|   x_i
\|^2    S_\gamma^t x_\gamma - \langle   x_i
,    x_\gamma\rangle   S_\gamma^t x_i
\right)+c_p(\hR_{q_\gamma \rightarrow q_i}(p_\gamma) -  p_i)+U_i
\\&=-\left[
\dot W_\gamma^tS_\gamma^tx_i+W_\gamma^tW_\gamma^tS_\gamma^tx_i+2W_\gamma^tS_\gamma^t\dot x_i\right]\\
&\quad-\left(\langle    q_i,W_\gamma^tW_\gamma^t   q_i\rangle -2\langle    q_i,W_\gamma^t  p_i\rangle + \|v_i\|^2\right) S_\gamma^t x_i
+ \sum_{j=1}^N\frac{\sigma_{ij}}{N}(\|    x_i\|^2   S_\gamma^t x_j
 - \langle   x_i
,  x_j
 \rangle  S_\gamma^t x_i
)\\
&\quad +c_q\Big(\|   x_i
\|^2    S_\gamma^t x_\gamma - \langle   x_i
,    x_\gamma\rangle   S_\gamma^t x_i
\Big)+c_p(\hR_{q_\gamma \rightarrow q_i}(p_\gamma) -  p_i)+U_i
.
\end{aligned}
\end{align*}
Note that
\begin{align*}
&-\left[
\dot W_\gamma^tS_\gamma^tx_i+W_\gamma^tW_\gamma^tS_\gamma^tx_i+2W_\gamma^tS_\gamma^t\dot x_i\right]-\Big(\langle    q_i,W_\gamma^tW_\gamma^t   q_i\rangle -2\langle    q_i,W_\gamma^t  p_i\rangle \Big) S_\gamma^t x_i
\\&
\qquad\qquad=-\left[
\dot W_\gamma^tq_i+W_\gamma^tW_\gamma^tq_i+2W_\gamma^t(p_i-W_\gamma q_i)\right]
-\Big(\langle    q_i,W_\gamma^tW_\gamma^t   q_i\rangle -2\langle    q_i,W_\gamma^t  p_i\rangle \Big) q_i
\\
&\qquad\qquad=-\langle    q_i,W_\gamma^tW_\gamma^t   q_i\rangle  q_i
+W_\gamma^tW_\gamma^t   q_i
+2\langle   q_i,W_\gamma^t   p_i\rangle  q_i
- 2W_\gamma^t  p_i
-\dot W_\gamma^t  q_i\\
&\qquad\qquad=-2\langle w_\gamma,  q_i\rangle (  q_i\times   p_i )-\dot w_\gamma\times   q_i.
\end{align*}
Therefore, by the property of $S_\gamma$ and the above two equalities,  we obtain that $\{(x_i(t),v_i(t))\}_{i=1}^N$ satisfies \eqref{eq structure} with \eqref{eq 3.65}.

\end{proof}

%
%{\color{red}
%Note that $W_\gamma  q_i=w_\gamma\times   q_i$ and
%}

\section{Reduction to a linearized system with a negative definite coefficient matrix}\label{sec4}
\setcounter{equation}{0}
In this section, we derive a linearized system from the structural system in  \eqref{eq structure}. We define auxiliary variables motivated by the flat case in Section \ref{sec2} and we extract leading order terms using $\|q_i(t)\|=1$ and $\langle q_i(t),p_i(t)\rangle =0$ for all $t\geq 0$ and $i\in \{1,\ldots,N\}$. In the system with respect to auxiliary variables,  leading order terms form an inhomogeneous linear system of ODEs with a negative definite coefficient matrix.

We consider the following system of ODEs with $\sigma_{ij}=\sigma>0$ and  $c_q,c_p>0$.
\begin{align}
\begin{aligned}\label{TR_eq}
\dot{x}_i&=v_i,\\
\dot{v}_i&=-\frac{\|v_i\|^2}{\|x_i\|^2}x_i + \sum_{j=1}^N\frac{\sigma}{N}(\|x_i\|^2 x_j - \langle x_i,x_j \rangle  x_i)\\
&\qquad\qquad\qquad\qquad\qquad+c_q(\|x_i\|^2 x_\gamma - \langle x_i,x_\gamma \rangle  x_i)
-c_pv_i+A_i.
\end{aligned}
\end{align}
For consistency, we additionally assume that for all $t\geq 0$,
\[\langle A_i(t),x_i(t)\rangle =0,\quad \mbox{for all} ~i\in \{1,\ldots,N\},\]
and the initial data satisfies
\[\|x_i(0)\|=1\quad \mbox{and}\quad  \langle v_i(0),x_i(0)\rangle =0,\quad \mbox{for all} ~i\in \{1,\ldots,N\}.\]

We now define the auxiliary variables as follows.
\begin{align*}
X_\gamma^1=\frac{1}{N}\sum_{i=1}^N\|x_i-x_\gamma\|^2,
\quad
 X_\gamma^2=\frac{1}{N}\sum_{i=1}^N\langle x_i-x_\gamma,v_i\rangle,
  \quad
  X_\gamma^3=\frac{1}{N}\sum_{i=1}^N\langle v_i,v_i\rangle,
\end{align*}
and
\begin{gather*}
X^1= \frac{1}{N^2}\sum_{i,k=1}^N\langle x_i-x_k,x_i-x_k\rangle, \quad
X^2= \frac{1}{N^2}\sum_{i,k=1}^N\langle v_i-v_k,x_i-x_k\rangle,\\
X^3= \frac{1}{N^2}\sum_{i,k=1}^N\langle v_i-v_k,v_i-v_k\rangle.
\end{gather*}
We also define the corresponding inhomogeneous terms as follows.
\begin{align*}
F_\gamma^1&=0,\\
 F_\gamma^2&=-\frac{1}{N}\sum_{i=1}^N\frac{\|v_i\|^2 }{2}\| x_i-x_\gamma\|^2
+\frac{\sigma}{4N^2}\sum_{i,j=1}^N\| x_i-x_j\|^2 \|x_i-x_\gamma\|^2
\\&\quad +\frac{c_q}{4N}\sum_{i=1}^N\|x_i-x_\gamma\|^4
+\frac{1}{N}\sum_{i=1}^N\langle x_i-x_\gamma,A_i\rangle,\\
  F_\gamma^3&=\frac{2}{N}\sum_{i=1}^N\langle v_i,A_i\rangle,
\end{align*}
and
\begin{align*}
F^1&= 0,\\
F^2&= -\frac{1}{N^2}\sum_{i,k=1}^N\frac{\|v_i\|^2+\|v_k\|^2}{2}\| x_i-x_k\|^2
+\frac{\sigma}{2N^3}\sum_{i,j,k=1}^N\|x_i-x_j\|^2\|x_i-x_k\|^2
\\&\quad +\frac{c_q}{2N^2}\sum_{i,k=1}^N  \|x_\gamma -  x_i\|^2\|x_i-x_k\|^2
+\frac{1}{N^2}\sum_{i,k=1}^N\langle A_i-A_k,x_i-x_k\rangle
,\\
F^3&= \frac{2}{N^2}\sum_{i,k=1}^N\left(\|v_i\|^2\langle x_i,v_k\rangle +\|v_k\|^2\langle x_x,v_i\rangle\right)+\frac{2\sigma}{N^3}\sum_{i,j,k=1}^N\|x_i-x_j\|^2\langle x_i,v_k\rangle
\\
&\quad +\frac{c_q}{N^2}\sum_{i,k=1}^N \|x_\gamma -  x_i\|^2\langle x_i,v_k\rangle
+\frac{2}{N^2}\sum_{i,k=1}^N\langle A_i-A_k,v_i-v_k\rangle.
\end{align*}

Let
\begin{gather}\label{F}
X=(X_\gamma^1,X_\gamma^2,X_\gamma^3,X^1,X^2,X^3)^T,\quad  F=(F_\gamma^1,F_\gamma^2,F_\gamma^3,F^1,F^2,F^3)^T.
\end{gather}
\begin{proposition}\label{prop 4.1}
For the auxiliary variable  $X$ and the inhomogeneous term $F$, the following holds.
\[\dot X=M X+F,\]
where the coefficient matrix $M$ is given by
\begin{align*}
M=
\begin{bmatrix}
0 &2 &0&0&0&0  \\
 -c_q & -c_p  &1&-\sigma/2&0&0 \\
0 &-2c_q & -2c_p&0&\sigma&0\\
0&0&0&0 &2 &0 \\
 0&0&0& -(c_q+\sigma) & -c_p  &1 \\
0&0&0&0 &-2(c_q+\sigma) & -2c_p\\
\end{bmatrix}.
\end{align*}
\end{proposition}

\begin{proof}
Clearly,
\begin{align*}
\frac{d}{dt}X_\gamma^1=2 X_\gamma^2.
\end{align*}
For $X_\gamma^2$, we have
\begin{align*}
\frac{d}{dt}X_\gamma^2&= X_\gamma^3+\frac{1}{N}\sum_{i=1}^N\langle x_i-x_\gamma,\dot v_i\rangle\\
&= X_\gamma^3+\frac{1}{N}\sum_{i=1}^N\bigg\langle x_i-x_\gamma,~     -\|v_i\|^2x_i + \sum_{j=1}^N\frac{\sigma}{N}(\|x_i\|^2 x_j - \langle x_i,x_j \rangle  x_i)\\
&\qquad\qquad\qquad\qquad\qquad\qquad+c_q(\|x_i\|^2 x_\gamma - \langle x_i,x_\gamma \rangle  x_i)-c_pv_i
+A_i\bigg\rangle
\\&=
 X_\gamma^3
-\frac{1}{N}\sum_{i=1}^N\frac{\|v_i\|^2 }{2}\| x_i-x_\gamma\|^2
+\frac{\sigma}{N^2}
\sum_{i,j=1}^N \langle x_i-x_\gamma,    x_j - \langle x_i,x_j \rangle  x_i\rangle
\\&\quad
+
\frac{c_q}{N}\sum_{i=1}^N \langle x_i-x_\gamma,   x_\gamma - \langle x_i,x_\gamma \rangle  x_i\rangle
-\frac{c_p}{N}\sum_{i=1}^N
 \langle x_i-x_\gamma,
 v_i
\rangle+\frac{1}{N}\sum_{i=1}^N\langle x_i-x_\gamma,A_i\rangle.
\end{align*}
Note that by $x_i\in \bbs^2$ and changing the indices,
\begin{align}\begin{aligned}\label{eq 5.1}
\sum_{i,j=1}^N \langle x_i-x_\gamma,    x_j - \langle x_i,x_j \rangle  x_i\rangle&=-\sum_{i,j=1}^N \langle x_\gamma,    x_j - \langle x_i,x_j \rangle  x_i\rangle
\\
&=-\sum_{i,j=1}^N \langle x_\gamma,    x_i - \langle x_i,x_j \rangle  x_i\rangle
\\
&=-\sum_{i,j=1}^N\frac{\|x_i-x_j\|^2}{2} \langle x_\gamma,      x_i\rangle
\\
&=-\frac{1}{2}\sum_{i,j=1}^N\| x_i-x_j\|^2 +\frac{1}{4}\sum_{i,j=1}^N\| x_i-x_j\|^2 \|x_i-x_\gamma\|^2.
\end{aligned}
\end{align}

By \eqref{eq 5.1},
we have
\begin{align*}
\frac{d}{dt}X_\gamma^2
&=
 X_\gamma^3
-\frac{1}{N}\sum_{i=1}^N\frac{\|v_i\|^2 }{2}\| x_i-x_\gamma\|^2-\frac{\sigma}{2}X^1+
\frac{\sigma}{4N^2}\sum_{i,j=1}^N\| x_i-x_j\|^2 \|x_i-x_\gamma\|^2
\\&\quad
-\frac{c_q}{N}\sum_{i=1}^N\|x_i-x_\gamma\|^2
+\frac{c_q}{4N}\sum_{i=1}^N\|x_i-x_\gamma\|^4-c_p
X_\gamma^2+\frac{1}{N}\sum_{i=1}^N\langle x_i-x_\gamma,A_i\rangle
\\
&=-c_qX_\gamma^1-c_p
X_\gamma^2
+
 X_\gamma^3-\frac{\sigma}{2}X^1+F_\gamma^2.\end{align*}
Similarly, we have
\begin{align*}
\frac{1}{2}\frac{d}{dt}X_\gamma^3&=\frac{1}{N}\sum_{i=1}^N\langle v_i,\dot v_i\rangle\\
&=\frac{1}{N}\sum_{i=1}^N \bigg\langle v_i,~     -\|v_i\|^2x_i + \sum_{j=1}^N\frac{\sigma}{N}(\|x_i\|^2 x_j - \langle x_i,x_j \rangle  x_i)\\
&\qquad \qquad \qquad \qquad +c_q(\|x_i\|^2 x_\gamma - \langle x_i,x_\gamma \rangle  x_i)-c_pv_i+A_i\bigg\rangle
\\&=
\frac{\sigma}{N^2}\sum_{i,j=1}^N \langle v_i,    x_j \rangle
-
\frac{1}{N}\sum_{i=1}^Nc_q \langle v_i,  x_i- x_\gamma\rangle
-\frac{1}{N}\sum_{i=1}^Nc_p
 \langle v_i,
 v_i
\rangle+\frac{1}{N}\sum_{i=1}^N\langle v_i,A_i\rangle
.\end{align*}
Thus, we have
\begin{align}\label{eq 5.3}
\frac{d}{dt}X_\gamma^3=-2c_qX^2_\gamma-2c_pX^3_\gamma-\sigma X^2+F_\gamma^3.
\end{align}

For $X^1$,
\[\frac{d}{dt}X^1=2X^2.\]
Similar to the previous cases, we use the second equation in \eqref{TR_eq} to obtain
\begin{align*}
\frac{d}{dt}X^2&=X^3+\frac{1}{N^2}\sum_{i,k=1}^N\langle \dot v_i-\dot v_k,x_i-x_k\rangle\\
&=X^3+\frac{1}{N^2}\sum_{i,k=1}^N\bigg\langle -\|v_i\|^2x_i+\|v_k\|^2x_k + \sum_{j=1}^N\frac{\sigma}{N}[ - \langle x_i,x_j \rangle  x_i+ \langle x_k,x_j \rangle  x_k]\\&\qquad\qquad\qquad\qquad\qquad+c_q\left[ - \langle x_i,x_\gamma \rangle  x_i+\langle x_k,~x_\gamma \rangle  x_k\right]
-c_pv_i+c_pv_k
,x_i-x_k\bigg\rangle\\
&\qquad +\frac{1}{N^2}\sum_{i,k=1}^N\langle A_i-A_k,x_i-x_k\rangle.
\end{align*}
By $x_i\in \bbs^2$, we have
\begin{align*}
\frac{d}{dt}X^2
&=X^3-\frac{1}{N^2}\sum_{i,k=1}^N\frac{\|v_i\|^2+\|v_k\|^2}{2}\| x_i-x_k\|^2
\\&\qquad-\sigma X^1
+\frac{\sigma}{4N^3}\sum_{i,j,k=1}^N\|x_i-x_j\|^2\|x_i-x_k\|^2
+\frac{\sigma}{4N^3}\sum_{i,j,k=1}^N\|x_k-x_j\|^2\|x_i-x_k\|^2
\\&\qquad-c_qX^1+\frac{c_q}{4N^2}\sum_{i,k=1}^N \|x_\gamma -  x_i\|^2\|x_i-x_k\|^2
+\frac{c_q}{4N^2}\sum_{i,k=1}^N   \|x_\gamma -  x_k\|^2\|x_i-x_k\|^2
\\&\qquad-c_p X^2+\frac{1}{N^2}\sum_{i,k=1}^N\langle A_i-A_k,x_i-x_k\rangle.
\end{align*}
Changing the indices implies that
\[\frac{d}{dt}X^2=-\sigma X^1-c_qX^1-c_p X^2+X^3+F^2.\]

Finally, for $X^3$, we obtain
\begin{align*}
\frac{1}{2}\frac{d}{dt}X^3&=\frac{1}{N^2}\sum_{i,k=1}^N\bigg\langle -\|v_i\|^2x_i+\|v_k\|^2x_k + \sum_{j=1}^N\frac{\sigma}{N}( - \langle x_i,x_j \rangle  x_i+ \langle x_k,x_j \rangle  x_k)\\&\qquad\qquad\qquad\qquad\qquad+c_q\left[ - \langle x_i,x_\gamma \rangle  x_i+\langle x_k,x_\gamma \rangle  x_k \right]
-c_pv_i+c_pv_k
,~v_i-v_k\bigg\rangle\\
&\qquad +\frac{1}{N^2}\sum_{i,k=1}^N\langle A_i-A_k,v_i-v_k\rangle\\
&=\frac{1}{N^2}\sum_{i,k=1}^N\left(\|v_i\|^2\langle x_i,v_k\rangle +\|v_k\|^2\langle x_x,v_i\rangle\right)
\\&\quad-\sigma X^2
+\sum_{i,j,k=1}^N\frac{\sigma}{2N^3}\|x_i-x_j\|^2\langle x_i,v_k\rangle
+\sum_{i,j,k=1}^N\frac{\sigma}{2N^3}\|x_k-x_j\|^2\langle x_k,v_i\rangle
\\&\quad-c_qX^2+\frac{c_q}{4N^2}\sum_{i,k=1}^N \|x_\gamma -  x_i\|^2\langle x_i,v_k\rangle
+\frac{c_q}{4N^2}\sum_{i,k=1}^N  \|x_\gamma -  x_k\|^2\langle x_k,v_i\rangle
-c_p X^3\\
&\quad +\frac{1}{N^2}\sum_{i,k=1}^N\langle A_i-A_k,v_i-v_k\rangle
.\end{align*}
Thus, we conclude that
\[\frac{d}{dt}X^3=-2\sigma X^2-2c_qX^2
-2c_p X^3+F^3.\]

\end{proof}

Note that the eigenvalues of the $6\times 6$ coefficient matrix $M$  have the only negative real part. The above result will be used for the complete rendezvous case.
\begin{remark}
In \cite{C-K-S2}, we use $l^\infty$-framework  to obtain a uniform decay estimate which is independent of $N$.
However, due to $X^2$ term on the right-hand side of \eqref{eq 5.3}, we cannot use this $l^\infty$-framework. We  obtain only the convergence result depending on $N$ by using the $6\times6$ system with $l^2$-framework.
\end{remark}

For the practical rendezvous result, we use a different framework, weighted $l^\infty$-framework. To obtain $l^\infty$-estimate, we define the following functionals:
\begin{align}\label{fnl}
X_i^1=\frac{4\|x_i-x_\gamma\|^2}{4-\|x_i-x_\gamma\|^2},
\quad
 X_i^2=\frac{16\langle x_i-x_\gamma,v_i\rangle}{\left(4-\|x_i-x_\gamma\|^2\right)^2},
  \quad
  X_i^3=\frac{16\langle v_i,v_i\rangle}{\left(4-\|x_i-x_\gamma\|^2\right)^2},
\end{align}
and
\begin{align}\begin{aligned}
\label{fnlf}
F_i^1&=0,\\
 F_i^2&=
-\frac{\|v_i\|^2 }{2}\frac{16\| x_i-x_\gamma\|^2}{\left(4-\|x_i-x_\gamma\|^2\right)^2}+\frac{16\sigma}{N\left(4-\|x_i-x_\gamma\|^2\right)^2}
\sum_{j=1}^N \langle x_i-x_\gamma,    x_j - \langle x_i,x_j \rangle  x_i\rangle
\\&\quad
+\frac{16\langle x_i-x_\gamma,A_i\rangle}{\left(4-\|x_i-x_\gamma\|^2\right)^2}
+\frac{64\langle x_i-x_\gamma, v_i\rangle^2}{\left(4-\|x_i-x_\gamma\|^2\right)^3}
   \\
  F_i^3&=\frac{32\sigma}{N\left(4-\|x_i-x_\gamma\|^2\right)^2}\sum_{j=1}^N \langle v_i,    x_j \rangle
+\frac{32\langle v_i,A_i\rangle}{\left(4-\|x_i-x_\gamma\|^2\right)^2}
+\frac{64\langle v_i, v_i\rangle\langle x_i-x_\gamma, v_i\rangle}{\left(4-\|x_i-x_\gamma\|^2\right)^3}.
\end{aligned}\end{align}
We note that due to the geometric structure of $\mathbb{S}^2$,  the quartic terms with the coefficient $c_q$  in $F_\gamma^2$ and $F^2$ appear. Thus, the standard functional $X(t)$ in the previous argument and Section \ref{sec2} does not work for this practical rendezvous case. For the complete rendezvous case, we will use the energy functional method and Lasalle’s invariance principle to control the quartic terms. However, for the practical rendezvous case, we cannot use the same methodology since the system is not autonomous. Thus, if an extra term with the coefficient $c_q$ appears in $F$, then it is hard to obtain the desired result.  Alternatively,  using the functionals in \eqref{fnl}, we can remove the quartic term with the coefficient $c_q$ as in \eqref{fnlf}.

By the same argument in Proposition \ref{eq 5.1}, we have
\begin{align*}
\frac{d}{dt}X_i^1=2 X_i^2.
\end{align*}
Using the second equation for the structural system, we obtain the following for $X_\gamma^2$.
\begin{align*}
\frac{d}{dt}X_i^2&= X_i^3+\frac{16\langle x_i-x_\gamma,\dot v_i\rangle}{\left(4-\|x_i-x_\gamma\|^2\right)^2}+\frac{64\langle x_i-x_\gamma, v_i\rangle^2}{\left(4-\|x_i-x_\gamma\|^2\right)^3}\\
&= X_i^3+16\bigg\langle x_i-x_\gamma,~     -\|v_i\|^2x_i + \sum_{j=1}^N\frac{\sigma}{N}(\|x_i\|^2 x_j - \langle x_i,x_j \rangle  x_i)\\
&\qquad\qquad\qquad\qquad\qquad\qquad+c_q(\|x_i\|^2 x_\gamma - \langle x_i,x_\gamma \rangle  x_i)-c_pv_i
+A_i\bigg\rangle/\left(4-\|x_i-x_\gamma\|^2\right)^2
\\&\quad+\frac{64\langle x_i-x_\gamma, v_i\rangle^2}{\left(4-\|x_i-x_\gamma\|^2\right)^3}\\
&=
 X_i^3
-\frac{\|v_i\|^2 }{2}\frac{16\| x_i-x_\gamma\|^2}{\left(4-\|x_i-x_\gamma\|^2\right)^2}
+\frac{16\sigma}{N\left(4-\|x_i-x_\gamma\|^2\right)^2}
\sum_{j=1}^N \langle x_i-x_\gamma,    x_j - \langle x_i,x_j \rangle  x_i\rangle
\\&\quad
+
\frac{16c_q}{\left(4-\|x_i-x_\gamma\|^2\right)^2}\langle x_i-x_\gamma,   x_\gamma - \langle x_i,x_\gamma \rangle  x_i\rangle
-\frac{16c_p}{\left(4-\|x_i-x_\gamma\|^2\right)^2}
 \langle x_i-x_\gamma,
 v_i
\rangle
\\&\quad+\frac{16\langle x_i-x_\gamma,A_i\rangle}{\left(4-\|x_i-x_\gamma\|^2\right)^2}+\frac{64\langle x_i-x_\gamma, v_i\rangle^2}{\left(4-\|x_i-x_\gamma\|^2\right)^3}.
\end{align*}
Note that
\begin{align*}
  \langle x_i-x_\gamma,   x_\gamma - \langle x_i,x_\gamma \rangle  x_i\rangle&= \langle x_i-x_\gamma,   x_\gamma \rangle- \langle x_i,x_\gamma \rangle \langle x_i-x_\gamma,   x_i\rangle\\
&= -\|x_i-x_\gamma\|^2- \langle x_i-x_\gamma,x_\gamma \rangle \langle x_i-x_\gamma,   x_i\rangle\\
&= -\|x_i-x_\gamma\|^2+\frac{\|x_i-x_\gamma\|^4}{4}.
\end{align*}
This implies that
\begin{align*}\begin{aligned}%\label{eq 6.3}
\frac{d}{dt}X_i^2
&=
 X_i^3
-\frac{\|v_i\|^2 }{2}\frac{16\| x_i-x_\gamma\|^2}{\left(4-\|x_i-x_\gamma\|^2\right)^2}+\frac{16\sigma}{N\left(4-\|x_i-x_\gamma\|^2\right)^2}
\sum_{j=1}^N \langle x_i-x_\gamma,    x_j - \langle x_i,x_j \rangle  x_i\rangle
\\&\quad
-c_qX_i^1
-c_p
X_i^2+\frac{16\langle x_i-x_\gamma,A_i\rangle}{\left(4-\|x_i-x_\gamma\|^2\right)^2}
+\frac{64\langle x_i-x_\gamma, v_i\rangle^2}{\left(4-\|x_i-x_\gamma\|^2\right)^3}
.\end{aligned}\end{align*}

For $X_i^3$, we have
\begin{align*}
\frac{d}{dt}X_\gamma^3&=\frac{32\langle v_i,\dot v_i\rangle}{\left(4-\|x_i-x_\gamma\|^2\right)^2}+\frac{64\langle v_i, v_i\rangle\langle x_i-x_\gamma, v_i\rangle}{\left(4-\|x_i-x_\gamma\|^2\right)^3}\\
&= 32\bigg\langle v_i,~     -\|v_i\|^2x_i + \sum_{j=1}^N\frac{\sigma}{N}(\|x_i\|^2 x_j - \langle x_i,x_j \rangle  x_i)\\
& \qquad \qquad +c_q(\|x_i\|^2 x_\gamma - \langle x_i,x_\gamma \rangle  x_i)-c_pv_i+A_i\bigg\rangle
/\left(4-\|x_i-x_\gamma\|^2\right)^2
+\frac{64\langle v_i, v_i\rangle\langle x_i-x_\gamma, v_i\rangle}{\left(4-\|x_i-x_\gamma\|^2\right)^3}
\\&=
\frac{32\sigma}{N\left(4-\|x_i-x_\gamma\|^2\right)^2}\sum_{j=1}^N \langle v_i,    x_j \rangle
-32c_q\frac{\langle v_i,  x_i- x_\gamma\rangle}{{\left(4-\|x_i-x_\gamma\|^2\right)^2}}
-32c_p
\frac{ \langle v_i,
 v_i
\rangle}{\left(4-\|x_i-x_\gamma\|^2\right)^2}\\
&\quad+\frac{32\langle v_i,A_i\rangle}{\left(4-\|x_i-x_\gamma\|^2\right)^2}
+\frac{64\langle v_i, v_i\rangle\langle x_i-x_\gamma, v_i\rangle}{\left(4-\|x_i-x_\gamma\|^2\right)^3}
\\&=
\frac{32\sigma}{N\left(4-\|x_i-x_\gamma\|^2\right)^2}\sum_{j=1}^N \langle v_i,    x_j \rangle
-2c_qX_i^2-2c_pX_i^3+\frac{32\langle v_i,A_i\rangle}{\left(4-\|x_i-x_\gamma\|^2\right)^2}
+\frac{64\langle v_i, v_i\rangle\langle x_i-x_\gamma, v_i\rangle}{\left(4-\|x_i-x_\gamma\|^2\right)^3}
.\end{align*}
In conclusion, we have
\begin{align*}
\frac{d}{dt}X_i^1&=2 X_i^2+F_i^1,\\
\frac{d}{dt}X_i^2
&=
 -c_qX^1_i-c_p
X_i^2+X_i^3+F_i^2
\\
\frac{d}{dt}X_i^3&=-2c_qX^2_i-2c_pX^3_i+F_i^3.
\end{align*}
Therefore, we have proved the following proposition.
\begin{proposition}\label{prop 4.3}
Let
\[X_i=(X_i^1,X_i^2,X_i^3)^T,\quad F_i=(F_i^1,F_i^2,F_i^3)^T,\]
where $X_i^k$, $F_i^k$, $k=1,2,3$ are functionals defined in \eqref{fnl} and \eqref{fnlf}.

Then the following holds.
\[\dot X_i=M_\infty X_i+F_i,\]
where the coefficient matrix $M_\infty$ is given by
\begin{align*}
M_\infty=
\begin{bmatrix}
0 &2 &0 \\
 -c_q & -c_p  &1\\
0 &-2c_q & -2c_p\end{bmatrix}.
\end{align*}
\end{proposition}

\section{Asymptotic analysis on the target tracking models: complete and practical rendezvouses }\label{sec5}
\setcounter{equation}{0}

In this section, we provide the proofs of Theorems \ref{thm1} and \ref{thm2} in Section \ref{sec1}.
Let $(q_\gamma,p_\gamma)$ be the phase of the target. We assume that the target satisfies \eqref{target} for some continuous $u_\gamma(t)\in \bbr^3$. For the given target $(q_\gamma(t),p_\gamma(t))$,  let  $\{(q_i(t),p_i(t))\}_{i=1}^N $ be the solution  to \eqref{main1}. By the argument in Section \ref{sec4}, we have the following equivalent system for $x_i(t)=S_\gamma^{-1}(t)p_i(t)$.
\begin{align}
\begin{aligned}\label{ST_eq}
\dot{x}_i&=v_i,\\
\dot{v}_i&=-\frac{\|v_i\|^2}{\|x_i\|^2}x_i + \sum_{j=1}^N\frac{\sigma_{ij}}{N}(\|x_i\|^2 x_j - \langle x_i,x_j \rangle  x_i)
+c_q(\|x_i\|^2 x_\gamma - \langle x_i,x_\gamma \rangle  x_i)
-c_pv_i+A_i,
\end{aligned}
\end{align}
where  $S_\gamma^t$ is the solution operator defined  by \eqref{eq W}-\eqref{eq W3}. For the angular velocity $w_\gamma=q_\gamma\times p_\gamma$, $A_i$ is the extra control law given by
\[A_i=S_\gamma^{-1}(t)U_i-2\langle w_\gamma,q_i\rangle S_\gamma^{-1}(t)[  q_i\times  p_i ]-S_\gamma^{-1}(t)[\dot w_\gamma(t)\times q_i]
.\]

\subsection{Complete rendezvouses} We assume that $\sigma_{ij}=\sigma>0$ and  $A_i=0$, i.e.,
\[U_i=2\langle w_\gamma,q_i\rangle   q_i\times  p_i -\dot w_\gamma(t)\times q_i.\]
We first use an energy functional method to obtain the convergence result in Theorem \ref{thm1} without convergence rate.  We now define an energy functional $\E=\E(\{(x_i,v_i)\}_{i=1}^N)$ as follows.
\[\E=\E_k+\E_c,\]
where $\E_k$ is the kinetic energy given by
\[\E_k=\frac{1}{2N}\sum_{i=1}^N\|v_i\|^2,\]
and $\E_c$ is the configuration energy given by
\[\E_c=\frac{\sigma}{4N^2}\sum_{i,j=1}^N   \|x_i-x_j\|^2+\frac{c_q}{2N}\sum_{i=1}^N\|x_\gamma-x_i\|^2.\]

This energy functional has a dissipation property. To obtain this, we take the inner product between $v_i$ and $\dot v_i$ to obtain
\begin{align*}
\frac{1}{2}\frac{d}{dt}\|v_i\|^2&=-\frac{\|v_i\|^2}{\|x_i\|^2}\langle x_i,v_i\rangle  + \sum_{j=1}^N\frac{\sigma}{N}(\|x_i\|^2\langle x_j,v_i\rangle  - \langle x_i,x_j \rangle \langle x_i,v_i\rangle )\\
&\quad\quad +c_q(\|x_i\|^2 \langle x_\gamma,v_i\rangle  - \langle x_i,x_\gamma \rangle  \langle x_i,v_i\rangle )
-c_p\langle v_i,v_i\rangle.\end{align*}
Using the orthogonality $\langle x_i,v_i\rangle=0$ and $\|x_i\|=1$ in \eqref{perp}, we have
\begin{align*}
\frac{1}{2}\frac{d}{dt}\|v_i\|^2&= \sum_{j=1}^N\frac{\sigma}{N} \langle x_j,v_i\rangle +c_q  \langle x_\gamma,v_i\rangle  -c_p\|v_i\|^2.
\end{align*}
Therefore,
\[\frac{d}{dt}\E_k= \sum_{i,j=1}^N\frac{\sigma}{N^2} \langle x_j,v_i\rangle +\frac{c_q}{N} \sum_{i=1}^N \langle x_\gamma,v_i\rangle -\frac{c_p}{N}\sum_{i=1}^N\|v_i\|^2.\]
Similarly,
\begin{align*}
\frac{d}{dt}\E_c&=\frac{\sigma}{2N^2}\sum_{i,j=1}^N  \langle  x_i-x_j, v_i-v_j\rangle
-\frac{c_q}{N}\sum_{i=1}^N\langle x_\gamma,v_i\rangle
 \\& =-\frac{\sigma}{2N^2}\sum_{i,j=1}^N  (\langle  x_i,v_j\rangle+\langle  x_j,v_i\rangle)
-\frac{c_q}{N}\sum_{i=1}^N\langle x_\gamma,v_i\rangle \\&
=
-\frac{\sigma}{N^2}\sum_{i,j=1}^N  \langle  x_j,v_i\rangle
-\frac{c_q}{N}\sum_{i=1}^N\langle x_\gamma,v_i\rangle.
\end{align*}
Therefore, we have
\[\frac{d}{dt}(\E_k+\E_c)=-\frac{c_q}{N}\sum_{i=1}^N\| v_i\|^2=-2c_q\E_k.\]

We notice that \eqref{ST_eq} is autonomous, since $x_\gamma$ is a constant vector. Moreover,  the energy functional $\E$ is zero if and only if
\[v_i=0~\mbox{for all}~i\in \{1,\ldots,N\}. \]
We can easily prove that the union of  the following two sets is the maximum invariant set of $\E$.
\[\left\{\{(x_i,v_i)\}_{i=1}^N: v_i=0, \quad x_i=x_\gamma ~\mbox{for all}~i\in \{1,\ldots,N\}\right\}\]
and
\[\bigg
\{\{(x_i,v_i)\}_{i=1}^N:v_i=0,\quad  \frac{\sigma}{N}\sum_{j=1}^N x_j+c_q x_\gamma=0  ~\mbox{for all}~i\in \{1,\ldots,N\}\bigg\}
.\]
  If we assume that $c_q>\sigma$ or $\displaystyle\E(0)<\frac{\sigma}{2}\left(1+\frac{c_q}{\sigma}\right)^2$, then $\displaystyle \frac{\sigma}{N}\sum_{j=1}^N x_j+c_q x_\gamma\ne 0$.  Thus,   by Lasalle’s invariance principle,
  \[\|v_i(t)\|\to 0 \quad \mbox{and}\quad  x_i(t)\to x_\gamma\]
  as $t\to \infty$.
Therefore, we have proved  the following proposition.
\begin{proposition}\label{prop 6.1}
 If $c_q>\sigma$ or $\displaystyle\E(0)<\frac{\sigma}{2}\left(1+\frac{c_q}{\sigma}\right)^2$, then
  \[v_i(t)\to 0\]
  and \[x_i(t)\to x_\gamma(t)\]
  as $t\to \infty$ for any initial data satisfying  $x_i(0)\ne-x_\gamma(0)$ for all $i\in \{1,\ldots,N\}$.
\end{proposition}
Next we consider exponential decay estimates for $\|x_i-x_\gamma\|$ and $\|v_i\|$. For notational simplicity, we define the following two functionals.
\[\mathcal{D}_x(t)=\max_{1\leq i\leq N}\|x_i(t)-x_\gamma(t)\|^2\]
and
\[\mathcal{D}_v(t)=\max_{1\leq i\leq N}\|v_i(t)\|^2.\]
\begin{proposition}\label{prop 6.2}
Assume that $A_i=0$. Then for the functional $F$ defined in \eqref{F},  the following estimate holds
\[\|F\|\leq 8(\sigma+c_q)[\mathcal{D}_x(t)+\mathcal{D}_v(t)]X^1_\gamma.\]
\end{proposition}
\begin{proof}
By elementary calculation, we have
\begin{align*}
| F_\gamma^1|&=0,\\
| F_\gamma^2|&\leq \left(\frac{\mathcal{D}_v(t)}{2}
+\frac{\sigma\mathcal{D}_x(t)}{4}+\frac{c_q\mathcal{D}_x(t)}{4}
\right)X^1_\gamma,\\
 | F_\gamma^3|&=0,
\end{align*}
and
\begin{align*}
F^1&= 0,\\
F^2&= -\frac{1}{N^2}\sum_{i,k=1}^N\frac{\|v_i\|^2+\|v_k\|^2}{2}\| x_i-x_k\|^2
+\frac{\sigma}{2N^3}\sum_{i,j,k=1}^N\|x_i-x_j\|^2\|x_i-x_k\|^2
\\&\quad +\frac{c_q}{2N^2}\sum_{i,k=1}^N  \|x_\gamma -  x_i\|^2\|x_i-x_k\|^2
+\frac{1}{N^2}\sum_{i,k=1}^N\langle A_i-A_k,x_i-x_k\rangle
,\\
F^3&= \frac{2}{N^2}\sum_{i,k=1}^N\left(\|v_i\|^2\langle x_i,v_k\rangle +\|v_k\|^2\langle x_x,v_i\rangle\right)+\frac{2\sigma}{N^3}\sum_{i,j,k=1}^N\|x_i-x_j\|^2\langle x_i,v_k\rangle
\\
&\quad +\frac{c_q}{N^2}\sum_{i,k=1}^N \|x_\gamma -  x_i\|^2\langle x_i,v_k\rangle
+\frac{2}{N^2}\sum_{i,k=1}^N\langle A_i-A_k,v_i-v_k\rangle.
\end{align*}

Note that
\begin{align}\begin{aligned}\label{eq 5.4}
\| x_i-x_k\|^2&\leq \| x_i-x_\gamma+x_\gamma-x_k\|^2\\
&\leq 2\| x_i-x_\gamma\|^2+2\|x_\gamma-x_k\|^2\\
&\leq 4\mathcal{D}_x(t),
\end{aligned}\end{align}
and
\begin{align}\begin{aligned}\label{eq 5.5}
  |\langle x_i,v_k\rangle|&=|\langle x_i-x_k,v_k\rangle|\\
&\leq |\langle x_i-x_\gamma,v_k\rangle|+|\langle x_\gamma-x_k,v_k\rangle|\\
&\leq \mathcal{D}_x(t)+\mathcal{D}_v(t).
\end{aligned}\end{align}
 By \eqref{eq 5.4} and \eqref{eq 5.5},
we have
\begin{align*}
|F^2|&\leq  \frac{\mathcal{D}_v(t)}{N^2}\sum_{i,k=1}^N\| x_i-x_k\|^2
+\frac{2\sigma\mathcal{D}_x(t)}{N^2}\sum_{i,k=1}^N\|x_i-x_k\|^2
 +\frac{c_q\mathcal{D}_x(t)}{2 N^2}\sum_{i,k=1}^N \|x_i-x_k\|^2,\end{align*}
 and
 \begin{align*}
  | F^3|&\leq 4(\mathcal{D}_x(t)+\mathcal{D}_v(t))\frac{1}{N}\sum_{i=1}^N\|v_i\|^2+  \frac{2\sigma(\mathcal{D}_x(t)+\mathcal{D}_v(t))}{N^2}\sum_{i,j=1}^N\|x_i-x_j\|^2
\\
&\quad +\frac{c_q(\mathcal{D}_x(t)+\mathcal{D}_v(t))}{ N}\sum_{i=1}^N  \|x_\gamma -  x_i\|^2.\end{align*}
Similarly, we have
\begin{align*}
\frac{1}{N^2}\sum_{i,k=1}^N\| x_i-x_k\|^2&=\frac{1}{N^2}\sum_{i,k=1}^N\| x_i-x_\gamma+x_\gamma-x_k\|^2\\&\leq 4X_\gamma^1.
\end{align*}
Therefore, we obtain that
\begin{align*}
|F^1|&= 0,\\
|F^2|&\leq  \left( 4\mathcal{D}_v(t) +8\sigma\mathcal{D}_x(t)+2c_q\mathcal{D}_x(t)\right)X^1_\gamma
,\\
|F^3|&\leq  \left(8\sigma+c_q\right)(\mathcal{D}_x(t)+\mathcal{D}_v(t))X^1_\gamma+4(\mathcal{D}_x(t)+\mathcal{D}_v(t))X_\gamma^3
.
\end{align*}
The above implies the result in this lemma.

\end{proof}

We are ready to prove Theorem \ref{thm1}.  We first check that the coefficient matrix $M$ has the following six eigenvalues.
\[\left\{-c_p,~-c_p,~-c_p\pm\sqrt{-4c_q+c_p^2},~-c_p\pm\sqrt{-4c_q+c_p^2-4\sigma}\right\}.\]
Thus, their real parts are all negative. Let $D$ be the greatest real part of the above eigenvalues and we define
\[\mu:=-D>0.\]
 Then by Proposition \ref{prop 6.1}, for any $\epsilon>0$, there is  $t_0>0$ such that if $t>t_0$, then
\[0\leq \mathcal{D}_x(t)+\mathcal{D}_v(t)<\frac{\epsilon}{4\left(1+2\sigma+2c_q\right)}.\]
From Proposition \ref{prop 4.1} and  \ref{prop 6.2}, it follows that
\begin{align*}
  X(t)&=e^{A(t-t_0)}X(t_0)+\int_{t_0}^te^{A(t-s)}F(s)ds.  \end{align*}
This implies that
\begin{align*}
  \|X(t)\|&\leq e^{-\mu (t-t_0)}\|X(t_0)\|+\int_{t_0}^te^{-\mu (t-s)}\|F(s)\|ds\\
  &\leq e^{-\mu (t-t_0)}\|X(t_0)\|+\epsilon \int_{t_0}^te^{-\mu (t-s)}\|X(s)\|ds.  \end{align*}

%
%
%\begin{align*}
%  \frac{1}{2}\frac{d}{dt}\|X\|^2&=\langle X,AX\rangle +\langle X,F\rangle
%  \\&
%  \leq
%   -\mu\| X\|^2+\epsilon\|X\|^2.
%  \end{align*}

Therefore, by the Gronwall inequality, if  $t>t_0$, then
\[\|X(t)\|\leq \|X(t_0)\|e^{-(\mu-\epsilon)(t-t_0)}.\]

\subsection{Practical rendezvouses}
In this part, we consider the  target tracking problem without acceleration information of the target. We assume that $\sigma_{ij}=\sigma>0$ and target speed and acceleration are bounded:
\[\|w_\gamma(t)\|,\|\dot w_\gamma(t)\|<C^w_\gamma,\quad t\geq 0,\]
where  $C^w_\gamma>0$ is a positive constant.
We assume that $U_i=0$. We first check that the coefficient matrix $M_\infty$ in Proposition \ref{prop 4.3} has the following eigenvalues.
\[\left\{-c_p,~-c_p\pm\sqrt{-4c_q+c_p^2}\right\}.\]
Thus, their real parts are all negative. Let $D_\infty$ be the greatest real part of the above eigenvalues and we define
\[\mu_\infty:=-D_\infty>0.\]

Let
\[X_\infty=\max_{1\leq i\leq N }\|X_i\|.\]
By Proposition \ref{prop 4.3}, for any fixed $t>0$, there is an index $i_t\in \{1,\ldots,N\}$ such that
\[\|X_{i_t}\|=X_{\infty}\]
and
\begin{align*}
\frac{d}{dt}X_{\infty}^2&=\frac{d}{dt}\|X_{i_t}\|^2\\
&=\langle X_{i_t}, M_\infty X_{i_t}\rangle +\langle X_{i_t}, F_{i_t}\rangle\\
&\leq -\mu_\infty\|X_{i_t}\|^2 +\| X_{i_t}\|\| F_{i_t}\|\\
&= -\mu_\infty X_\infty^2 +X_\infty\| F_{i_t}\|
.
\end{align*}
By direct calculation,
\begin{align*}
|F_i^1|&=0,\\
| F_i^2|&\leq
\frac{\|v_i\|^2 }{2}\frac{16\| x_i-x_\gamma\|^2}{\left(4-\|x_i-x_\gamma\|^2\right)^2}+\frac{16\sigma}{N\left(4-\|x_i-x_\gamma\|^2\right)^2}
\sum_{j=1}^N |\langle x_i-x_\gamma,    x_j - \langle x_i,x_j \rangle  x_i\rangle|
\\&\quad
+\frac{16|\langle x_i-x_\gamma,A_i\rangle|}{\left(4-\|x_i-x_\gamma\|^2\right)^2}
+\frac{64\langle x_i-x_\gamma, v_i\rangle^2}{\left(4-\|x_i-x_\gamma\|^2\right)^3}
\\  |F_i^3|&\leq \frac{32\sigma}{N\left(4-\|x_i-x_\gamma\|^2\right)^2}\sum_{j=1}^N |\langle v_i,    x_j \rangle|
+\frac{32|\langle v_i,A_i\rangle|}{\left(4-\|x_i-x_\gamma\|^2\right)^2}
+\frac{64\langle v_i, v_i\rangle|\langle x_i-x_\gamma, v_i\rangle|}{\left(4-\|x_i-x_\gamma\|^2\right)^3}.
\end{align*}
We note that
\begin{align*}
\langle x_i-x_\gamma,    x_j - \langle x_i,x_j \rangle  x_i\rangle&=\langle x_i-x_\gamma,    x_j -x_\gamma\rangle+\langle x_i-x_\gamma,    x_\gamma \rangle-\langle x_i-x_\gamma,    \langle x_i,x_j \rangle  x_i\rangle\\
&=\langle x_i-x_\gamma,    x_j -x_\gamma\rangle-\frac{1}{2}\|x_i-x_\gamma\|^2- \frac{ \langle x_i,x_j \rangle }{2}\| x_i-x_\gamma\|^2.
\end{align*}
This implies that
\[ |\langle x_i-x_\gamma,    x_j - \langle x_i,x_j \rangle  x_i\rangle|\leq 2\max_{1\leq i\leq N}\|x_i-x_\gamma\|^2 .\]
Similarly,
\[|\langle v_i,    x_j \rangle|=|\langle v_i,   x_j- x_i \rangle|\leq |\langle v_i,   x_j-x_\gamma \rangle|+|\langle v_i,   x_\gamma-x_i \rangle|\leq \|v_i\|^2+\max_{1\leq i\leq N}\|x_i-x_\gamma\|^2,\]
\[\langle x_i-x_\gamma, v_i\rangle^2\leq 4\|v_i\|^2.\]
Thus,
\begin{align*}
|F_i^1|&=0,\\
| F_i^2|&\leq
2 X_\infty+\frac{32\sigma\max_{1\leq i\leq N}\|x_i-x_\gamma\|^2}{\left(4-\|x_i-x_\gamma\|^2\right)^2}
+\frac{16|\langle x_i-x_\gamma,A_i\rangle|}{\left(4-\|x_i-x_\gamma\|^2\right)^2}
+\frac{256\|v_i\|^2}{\left(4-\|x_i-x_\gamma\|^2\right)^3},
  \\
  |F_i^3|&\leq 2\sigma X_\infty
+\frac{32\sigma\max_{1\leq i\leq N}\|x_i-x_\gamma\|^2}{\left(4-\|x_i-x_\gamma\|^2\right)^2}+\frac{32|\langle v_i,A_i\rangle|}{\left(4-\|x_i-x_\gamma\|^2\right)^2}
+\frac{256\|v_i\|^3}{\left(4-\|x_i-x_\gamma\|^2\right)^3}.
\end{align*}

By elementary calculation, we have
\begin{align*}
\left|\langle x_i-x_\gamma,A_i\rangle\right|\leq \| x_i-x_\gamma\|^2+ \frac{\|A_i\|^2}{4 }.
\end{align*}
Note that
\[\|A_i\|^2\leq 6\|w_\gamma\|^2      \| S_\gamma^{-1}(t)  p_i\|^2+3\|\dot w_\gamma\|^2.\]
Since $ p_i(t)=W_\gamma^t S_\gamma^tx_i(t)+ S_\gamma^t\dot x_i(t)$,
\[ \| S_\gamma^{-1}(t)  q_i\|^2\leq 2\|w_\gamma\|^2+2\|v_i\|^2\]
and
\[\|A_i\|^2\leq 12(C^w_\gamma)^2\|v_i\|^2+12(C^w_\gamma)^4+3(C^w_\gamma)^2.\]
Therefore, we have
\begin{align}\label{eq 6.1}
\left|\langle x_i-x_\gamma,A_i\rangle\right|\leq \| x_i-x_\gamma\|^2+ 3(C^w_\gamma)^2\|v_i\|^2
+3(C^w_\gamma)^4+ \frac{3(C^w_\gamma)^2}{4}.
\end{align}

Similarly, we have
\begin{align}\label{eq 6.2}
\left|\langle v_i,A_i \rangle\right|\leq \| v_i\|^2+ 3(C^w_\gamma)^2\|v_i\|^2
+3(C^w_\gamma)^4+ \frac{3(C^w_\gamma)^2}{4}.
\end{align}
By \eqref{eq 6.1}-\eqref{eq 6.2} and the above argument, if $\displaystyle \max_{1\leq i\leq N}\|x_i-x_\gamma\|< \frac{2\sqrt{C_1-1}}{\sqrt{C_1}}<2$, then
\begin{align*}
|F_i^1|&=0,\\
| F_i^2|&\leq (2+2\sigma C_1+5C_1+3(C^w_\gamma)^2)X_\infty
+3(C^w_\gamma)^4C_1^2+ \frac{3(C^w_\gamma)^2C_1^2}{4}
,
  \\
  |F_i^3|&\leq \left(2+2\sigma+2\sigma C_1+6(C^w_\gamma)^2\right) X_\infty
+6(C^w_\gamma)^4C_1^2+ \frac{6(C^w_\gamma)^2C_1^2}{4}+4X_\infty^{3/2}.
\end{align*}
We conclude that
\[\|F_i\|\leq \left(4+2\sigma+ 4\sigma C_1+5C_1+9(C^w_\gamma)^2\right)X_\infty+9(C^w_\gamma)^4C_1^2+ \frac{9(C^w_\gamma)^2C_1^2}{4}+4X_\infty^{3/2}.\]
Therefore, we obtain that
\begin{align}\label{eq 6.4}
\dot X_\infty\leq -\mu_\infty X_\infty+\left(4+2\sigma+ 4\sigma C_1+5C_1+9(C^w_\gamma)^2\right)X_\infty+9(C^w_\gamma)^4C_1^2+ \frac{9(C^w_\gamma)^2C_1^2}{4}+4X_\infty^{3/2}.
\end{align}
We choose $c_q$ and $c_p$ sufficiently large and take
\[X_{\infty}(0)=\frac{\sqrt{C_1-1}}{\sqrt{C_1}}.\]
Let $T>0$ be a maximal number such that on $t\in [0,T)$,
\begin{align}\label{eq 5.7}
\max_{1\leq i\leq N}\|x_i(t)-x_\gamma(t)\|< 2X_{\infty}(0),\quad \mbox{$t\in [0,T)$.}
\end{align}
By the initial condition and the continuity of the solution, there is a positive number $T>0$ satisfying \eqref{eq 5.7}. We claim that if $c_q$ and $c_p$ are sufficiently large, then $T=\infty$. We note that for a given initial data, $\sigma$, $C_1$, $C_\gamma^w$ are fixed constants. Therefore, on $t\in [0,T)$,
\begin{align}\label{eq 6.4}
\dot X_\infty\leq -\mu_\infty X_\infty+C X_\infty+C.
\end{align}
\eqref{eq 6.4} implies
\begin{align}\label{eq 6.5}
\dot X_\infty\leq -\frac{\mu}{2} X_\infty+C,
\end{align}
if $c_q$ and $c_p$ are sufficiently large.
Therefore, by the Gronwall inequality and \eqref{eq 6.5},
\[ X_\infty(t)\leq e^{-\frac{\mu}{2} t}X_\infty(0)+
e^{-\frac{\mu}{2} t}\frac{2Ce^{\frac{\mu}{2} t}-2C}{\mu}=e^{-\frac{\mu}{2} t}\left(X_\infty(0)-\frac{2C}{\mu}\right)+\frac{2C}{\mu}
.\]
If $c_q$ and $c_p$  are sufficiently large, then $\mu$ is sufficiently large and  $X_{\infty}\leq X_\infty(0)$. These imply that on $t\in [0,T)$,
\[\max_{1\leq i\leq N}\|x_i(t)-x_\gamma(t)\|\leq X_{\infty}\leq X_\infty(0)< 2X_{\infty}(0).\]
By the continuity of the solution, we obtain that
 \[T=\infty.\]
Finally, by the above, we obtain the following practical rendezvous estimate.
\[ X_\infty(t)\leq e^{-\frac{\mu}{2} t}\left(X_\infty(0)-\frac{2C}{\mu}\right)+\frac{2C}{\mu}
.\]
Thus, we complete the proof of Theorem \ref{thm2}.

\section{Simulation results}\label{sec6}
\setcounter{equation}{0}
In this section, we present several numerical simulations for the target tracking problem on the unit sphere and the flat space to verify the asymptotic complete rendezvous and practical rendezvous. We use the fourth-order Runge-Kutta method. We consider six $\alpha$-agents $\{(q_i, p_i)\}_{i=1}^6$ chasing one target  $(q_\gamma, p_\gamma)$. We assume that the control law for the target $(q_\gamma,p_\gamma)$ is given by
\[u_\gamma(t)=a(\cos t, \sin t,1),\]
where $a$ is a constant.
Throughout this section, we assume that the inter-particle bonding force parameter is given by
\[\sigma=1.\]

With the extra control law for agents
\[U_i=2\langle w_\gamma,  q_i\rangle (  q_i\times   p_i )+\dot w_\gamma(t)\times   q_i,\] the initial positions and velocities for the agents are randomly chosen in
\[(q_i(0), p_i(0))\in T\mathbb{S}^2\cap[-1,1]^3\times [-1,1]^3\] as follows:
 \phantom{-}
\begin{align*}
&q_1(0)=( \phantom{-}0.8132, \phantom{-}0.4989, -0.2993), &&q_2(0)=( \phantom{-}0.7198, \phantom{-}0.4908, \phantom{-}0.4908),\\
&q_3(0)=(-0.6758, -0.6991, \phantom{-}0.2330), &&q_4(0)=(-0.7878, \phantom{-}0.5627, -0.2501 ), \\
&q_5(0)=( -0.5440, -0.7504, \phantom{-}0.3752),&& q_6(0)=(-0.8599,  -0.3608, \phantom{-}0.3608),
\end{align*}
and
\begin{align*}
&p_1(0)=(\phantom{-}0.1028, -0.1884, -0.0347), &&p_2(0)=( -0.1168, \phantom{-}0.5118, -0.3405 ),\\
&p_3(0)=(-0.0821, \phantom{-}0.0857, \phantom{-}0.0191), &&p_4(0)=( -0.1454, -0.1506, \phantom{-}0.1189 ), \\
&p_5(0)=( \phantom{-}0.2220, -0.1040, \phantom{-}0.1137 ),&& p_6(0)=(-0.0003, \phantom{-}0.3768, \phantom{-}0.3759).
\end{align*}
The initial data for the target is
\[q_\gamma(0)=(-0.6451, 0.6605, -0.3840) \quad\text{and}\quad p_\gamma(0)=(0.1761, 0.3646, 0.3311).\]

\begin{figure}[!ht]
\begin{minipage}{0.23\textwidth}
\centering
\includegraphics[width=\textwidth]{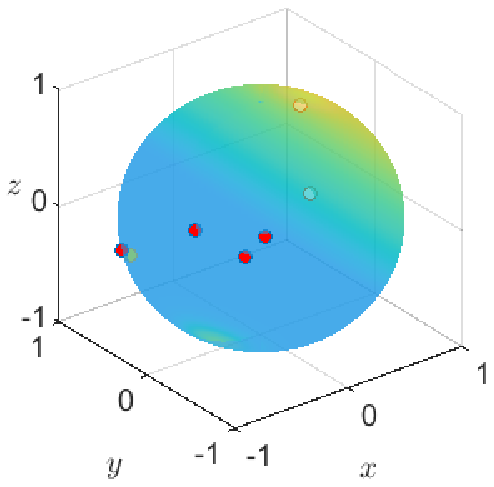}\\
(A)  $t=0$
\end{minipage}
\begin{minipage}{0.23\textwidth}
\centering
\includegraphics[width=\textwidth]{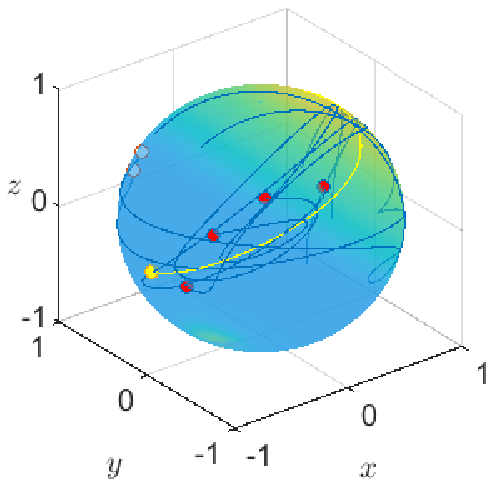}\\
(B) $t=5$
\end{minipage}
\begin{minipage}{0.23\textwidth}
\centering
\includegraphics[width=\textwidth]{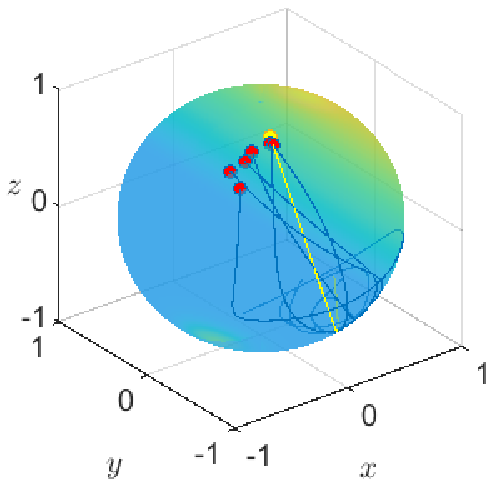}\\
(C) $t=25$
\end{minipage}
\begin{minipage}{0.23\textwidth}
\centering
\includegraphics[width=\textwidth]{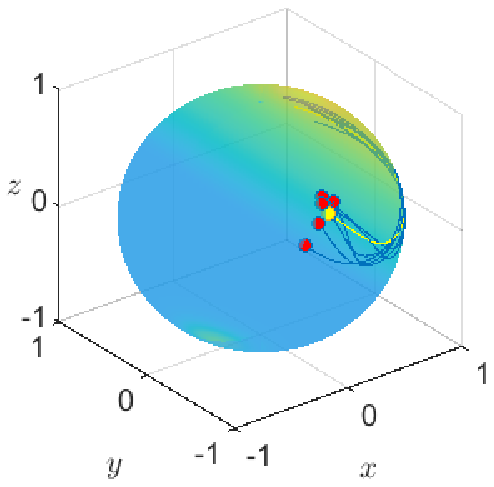}\\
(D) $t=40$
\end{minipage}
\begin{minipage}{0.23\textwidth}
\centering
\includegraphics[width=\textwidth]{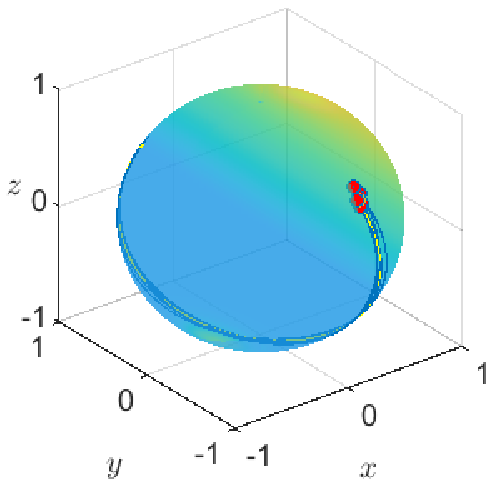}\\
(E) $t=55$
\end{minipage}
\begin{minipage}{0.23\textwidth}
\centering
\includegraphics[width=\textwidth]{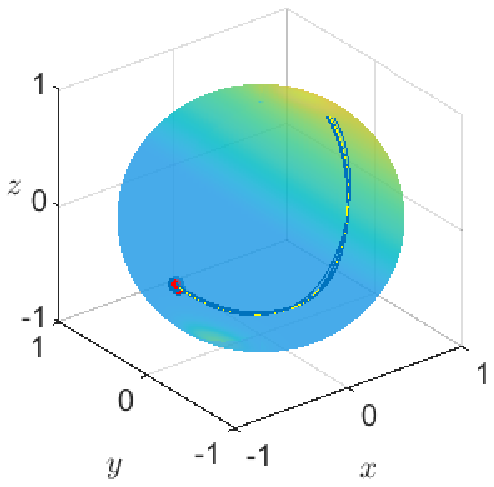}\\
(F) $t=70$
\end{minipage}
\begin{minipage}{0.23\textwidth}
\centering
\includegraphics[width=\textwidth]{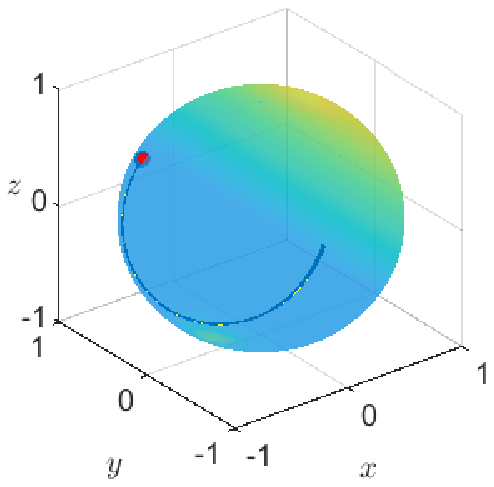}\\
(G) $t=100$
\end{minipage}
\begin{minipage}{0.23\textwidth}
\centering
\includegraphics[width=\textwidth]{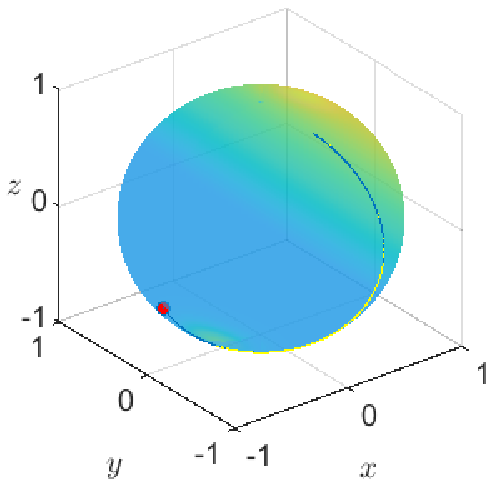}\\
(H) $t=200$
\end{minipage}
\caption{The time evolution of \eqref{main1} with extra control law \eqref{eq 6.0}}
\label{fig1}
\end{figure}

Note that all the initial positions and velocities   satisfy the admissible conditions in \eqref{initial}.
Since $\omega_\gamma=q_\gamma\times p_\gamma$, we can check that
\begin{align}
\begin{aligned}\label{eq 6.0}
U_i&=2\langle \omega_\gamma, q_i\rangle (q_i\times p_i)+\dot{\omega}_\gamma(t)\times q_i\\
&=2\langle \omega_\gamma, q_i\rangle (q_i\times p_i)+(\dot{q}_\gamma\times p_\gamma+q_\gamma\times\dot{p}_\gamma)\times q_i\\
&=2\langle \omega_\gamma, q_i\rangle (q_i\times p_i)+\left(q_\gamma\times \left[-\frac{\|p_\gamma\|^2}{\|q_\gamma\|^2}q_\gamma+\|q_\gamma\|^2u_\gamma-\langle u_\gamma ,q_\gamma\rangle q_\gamma\right]\right)\times q_i\\
&=2\langle \omega_\gamma, q_i\rangle (q_i\times p_i)+(q_\gamma\times\|q_\gamma\|^2u_\gamma)\times q_i.
\end{aligned}\end{align}
We fix
\[\sigma=1,~ c_q=5, ~c_p=0.1\quad \mbox{ and} \quad  a=0.5.\]
For this case,  the time evolution of \eqref{main1} is given in Figure \ref{fig1}. The red points and blue lines stand for the position $q_i(t)$ at $t=t_0$ and trajectories for the time interval $[t_0-3,t_0]$, respectively.  The yellow one is for the target agent $q_\gamma(t)$.
In addition, we can check that the asymptotic complete rendezvous occurs as we proved in Theorem \ref{thm1}. See Figure \ref{fig2}.
Here, the exponential function is $2e^{(-c_p+0.05)(t-40)}$.

%time evolution of complete 랑데뷰

%complete 랑데뷰의 distance and logplot
\begin{figure}[!ht]
\begin{minipage}{0.3\textwidth}
\centering
\includegraphics[width=\textwidth]{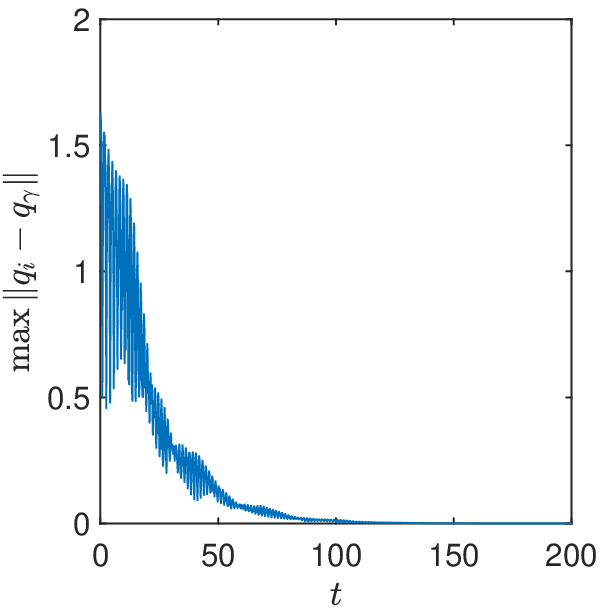}\\
(A)  $\max\limits_{1\le i\le 6}\|q_i(t)-q_\gamma\|$
\end{minipage}
\begin{minipage}{0.3\textwidth}
\centering
\includegraphics[width=\textwidth]{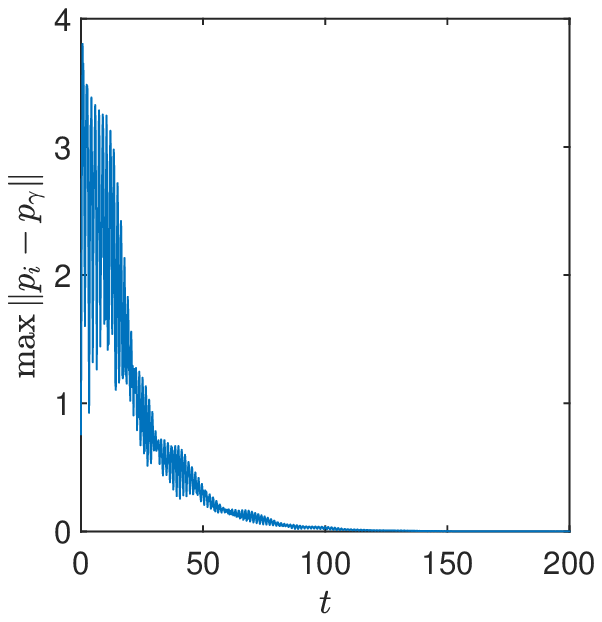}\\
(B) $\max\limits_{1\le i\le 6}\|p_i(t)-p_\gamma\|$
\end{minipage}
\begin{minipage}{0.3\textwidth}
\centering
\includegraphics[width=\textwidth]{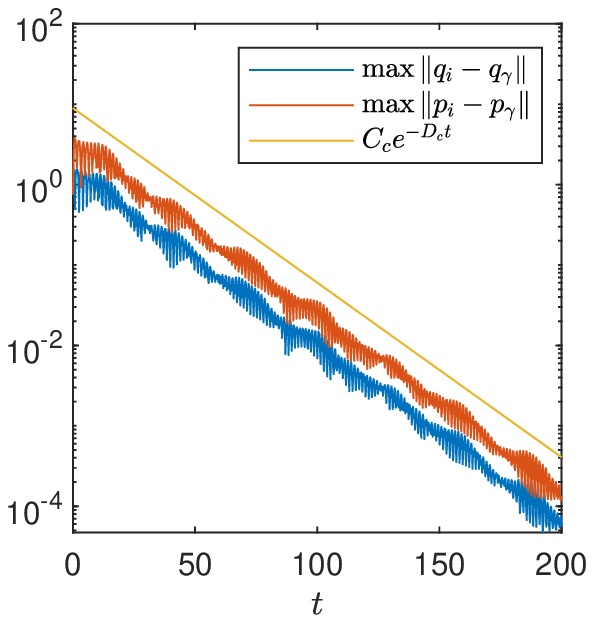}\\
(C) Logplot of (A) and (B)
\vspace{0.6em}
\end{minipage}
\caption{The asymptotic complete rendezvous}
\label{fig2}
\end{figure}

For the zero extra control law, i.e. $U_i=0$, we fix the parameters such that
\[\sigma=1, ~c_q=4,~ c_p=4,~ a=0.5.\]
The initial data of agents are randomly chosen but near the target as follows:
\begin{align*}
&q_1(0)=( -0.8147, -0.5366,  \phantom{-}0.2193), &&q_2(0)=( -0.4575, -0.8843,  \phantom{-}0.0922 ),\\
&q_3(0)=( -0.4335, -0.8173,  \phantom{-}0.3794), &&q_4(0)=( -0.8645, -0.2373,  \phantom{-}0.4429), \\
&q_5(0)=( -0.4420,  -0.7998, \phantom{-}0.4060),&& q_6(0)=( -0.4312, -0.6004, \phantom{-}0.6734  ),
\end{align*}
and
\begin{align*}
&p_1(0)=( 0.0228, -0.0750,  -0.0987), &&p_2(0)=(0.2519,-0.1263,   \phantom{-}0.0383 ),\\
&p_3(0)=(  0.0200, \phantom{-}0.0169,  \phantom{-}0.0594), &&p_4(0)=( 0.0388, -0.1447,   -0.0017 ), \\
&p_5(0)=( 0.0365,  \phantom{-}0.1109,   \phantom{-} 0.2583 ),&& p_6(0)=( 0.0081,  \phantom{-} 0.0050,  \phantom{-}0.0097  ).
\end{align*}
\begin{figure}[!ht]
\begin{minipage}{0.23\textwidth}
\centering
\includegraphics[width=\textwidth]{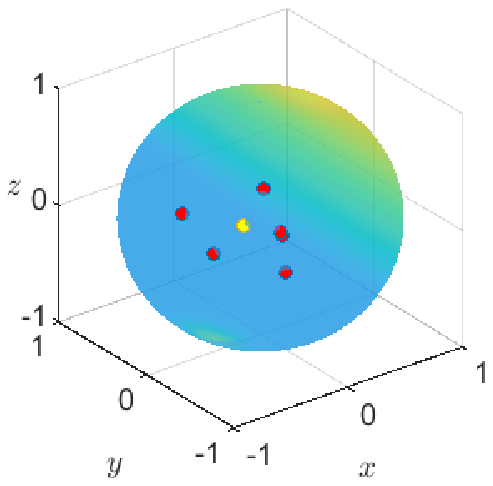}\\
(A)  $t=0$
\end{minipage}
\begin{minipage}{0.23\textwidth}
\centering
\includegraphics[width=\textwidth]{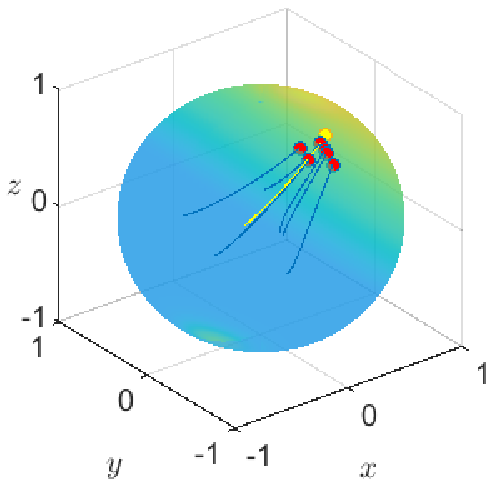}\\
(B) $t=1$
\end{minipage}
\begin{minipage}{0.23\textwidth}
\centering
\includegraphics[width=\textwidth]{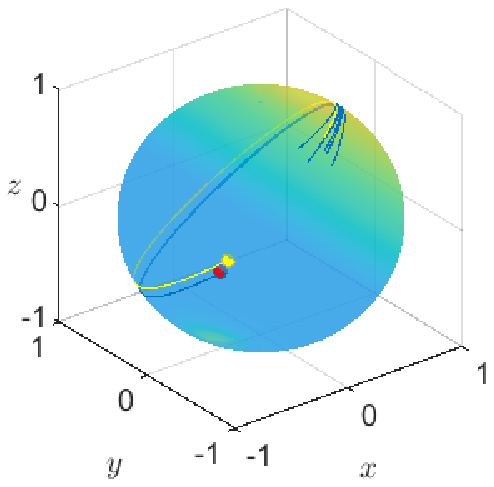}\\
(C) $t=4$
\end{minipage}
\begin{minipage}{0.23\textwidth}
\centering
\includegraphics[width=\textwidth]{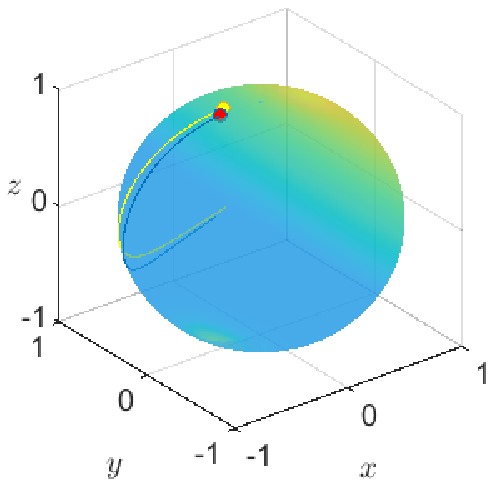}\\
(D) $t=9$
\end{minipage}
\begin{minipage}{0.23\textwidth}
\centering
\includegraphics[width=\textwidth]{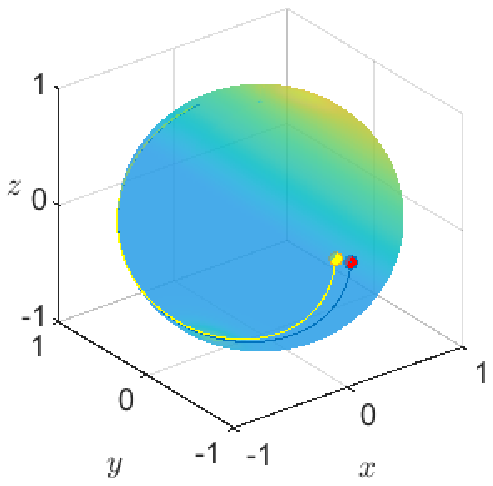}\\
(E) $t=35$
\end{minipage}
\begin{minipage}{0.23\textwidth}
\centering
\includegraphics[width=\textwidth]{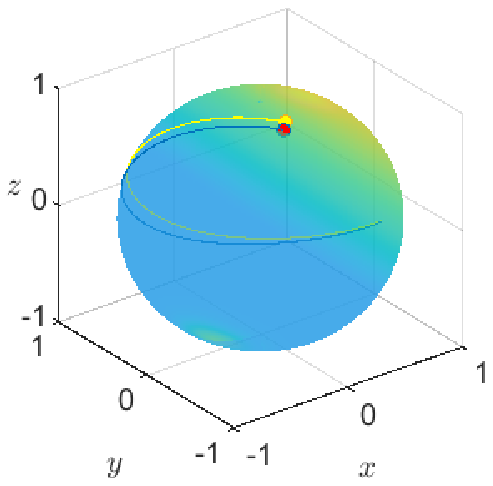}\\
(F) $t=90$
\end{minipage}
\begin{minipage}{0.23\textwidth}
\centering
\includegraphics[width=\textwidth]{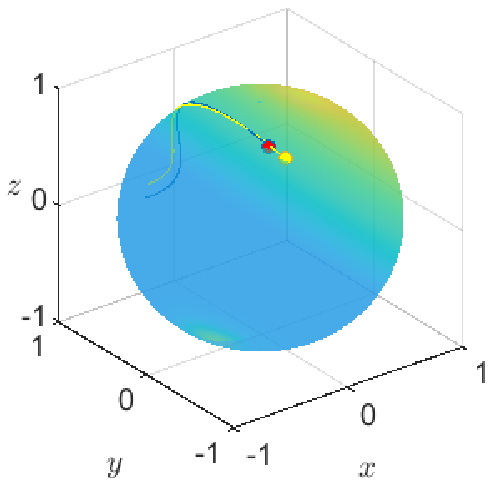}\\
(G) $t=150$
\end{minipage}
\begin{minipage}{0.23\textwidth}
\centering
\includegraphics[width=\textwidth]{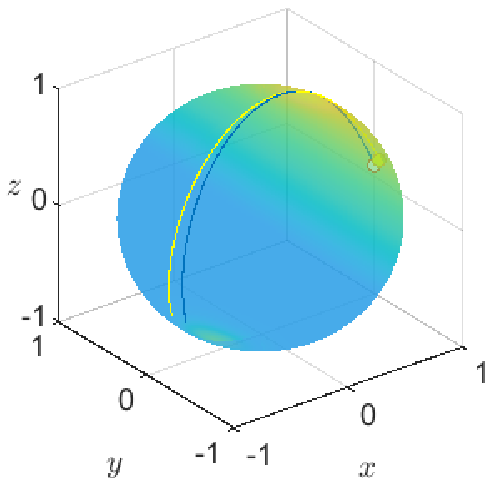}\\
(H) $t=200$
\end{minipage}
\caption{The time evolution of \eqref{main1} with control law}
\label{fig3}
\end{figure}
The initial data for the target is given by
\[q_\gamma(0)=(-0.6324, -0.6324, 0.4472) \quad\text{and}\quad p_\gamma(0)=(0.4712,-0.1742, 0.4199) .\]
Figure \ref{fig3} shows the time evolution of \eqref{main1} without extra control law.
%time evolution of practical 랑데뷰

We can see that the maximum distance
\[\max\limits_{1\le i\le 6}\|q_i(t)-q_\gamma(t)\|\] between agents and the target  is bounded by $2/\sqrt{c_p}$.
See Figure \ref{fig4}(A).
Let
\[d(t)=\max\limits_{1\le i\le 6}\|q_i(t)-q_\gamma(t)\|.\]
Figure \ref{fig4}(B) displays $d(t)$ at $t=100$ with respect to $c_p$. As $c_p$ increases,  the maximum distance between agents and target decreases.
Therefore, we observe that the asymptotic practical rendezvous occurs.
% d(t) when t=100일때 c_p에 따라 그림그리

\begin{figure}[!ht]
\begin{minipage}{0.3\textwidth}
\centering
\includegraphics[width=\textwidth]{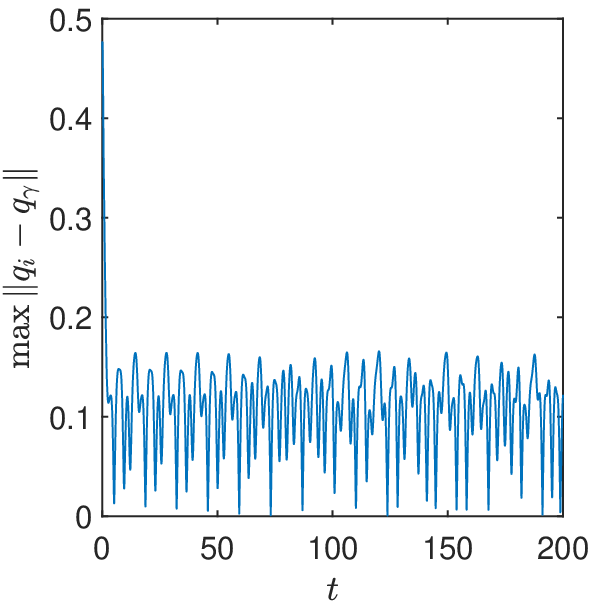}
(A) $\max\limits_{1\le i\le 6}\|q_i(t)-q_\gamma(t)\|$
\end{minipage}
\begin{minipage}{0.3\textwidth}
\centering
\includegraphics[width=\textwidth]{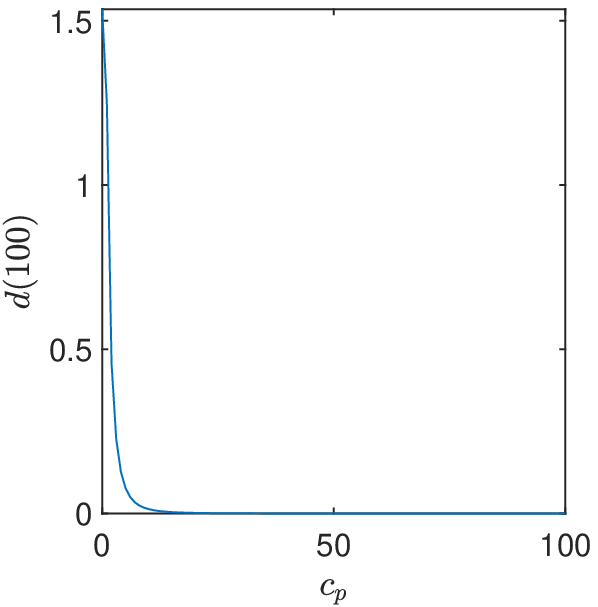}
(B)  $d(t)$ at $t=100$
\end{minipage}
\caption{The asymptotic practical rendezvous }
\label{fig4}
\end{figure}

With the extra control law, we observed the asymptotic complete rendezvous in Figure \ref{fig1} and Figure \ref{fig2}. However, if we choose the parameter $c_p$ as zero, then the agents are not able to track the target. See Figure \ref{fig9}. Here, other parameters and initial data are the same as the case in Figure \ref{fig1}. In the absence of the velocity alignment term, the agents easily escape the sphere due to the  accumulation of errors. To overcome this, as in \cite{C-K-S2}, we add the following feedback term $f_i^0$ on the second equation of \eqref{main1}.
\[f_i^0=-k_0\Big(q_i-\frac{q_i}{\|q_i\|}\Big),\]
where $k_0=10^4$.
From this, we conclude that the velocity alignment operator is crucial in this target tracking algorithm.

\begin{figure}[!ht]
\begin{minipage}{0.23\textwidth}
\centering
\includegraphics[width=\textwidth]{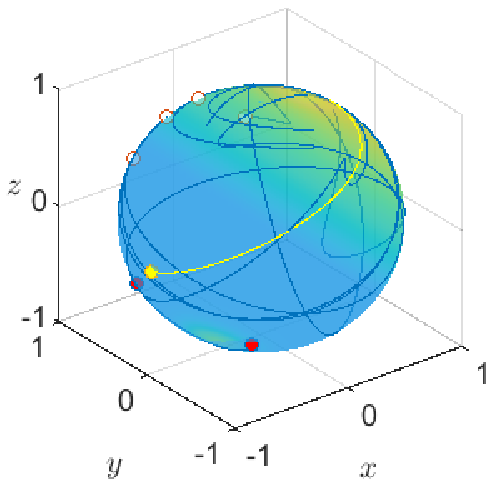}\\
(A)  $t=5$
\end{minipage}
\begin{minipage}{0.23\textwidth}
\centering
\includegraphics[width=\textwidth]{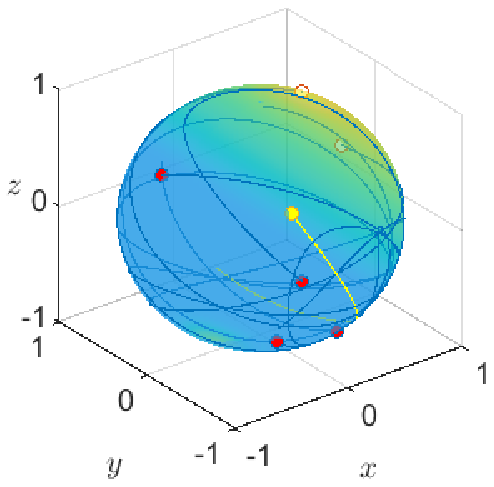}\\
(B) $t=55$
\end{minipage}
\begin{minipage}{0.23\textwidth}
\centering
\includegraphics[width=\textwidth]{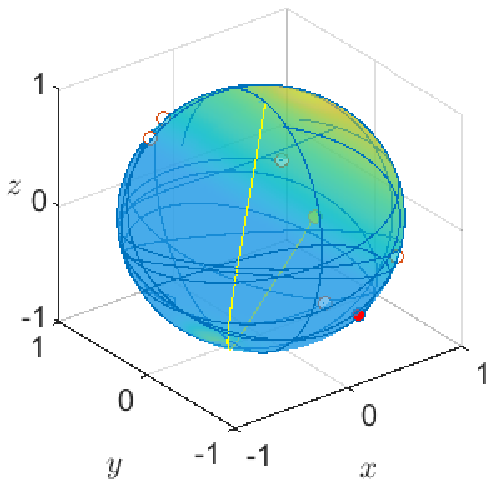}\\
(C) $t=100$
\end{minipage}
\begin{minipage}{0.23\textwidth}
\centering
\includegraphics[width=\textwidth]{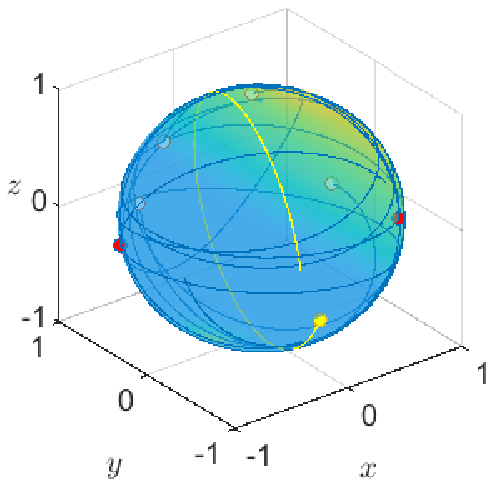}\\
(D) $t=200$
\end{minipage}
\caption{The time evolution of \eqref{main1} with extra control law \eqref{eq 6.0} and $c_p=0$}
\label{fig9}
\end{figure}

As we mentioned in Subsection \ref{sec 2.3}, the flocking term is negligible for the target tracking problem \eqref{main1}. With the same parameters of Figure \ref{fig1} and Figure \ref{fig3}, the numerical results of \eqref{main1} including the rotational flocking term \[\sum_{j=1}^N \frac{\psi_{ij}}{N}(R_{q_j\to q_i}(p_j)-p_i),\] where $\psi_{ij}=1$ is given in Figure \ref{fig5}. It is confirmed that the flocking term does not affect the results. See also Figure \ref{fig6}.

\begin{figure}[!ht]
\begin{minipage}{0.3\textwidth}
\centering
\includegraphics[width=\textwidth]{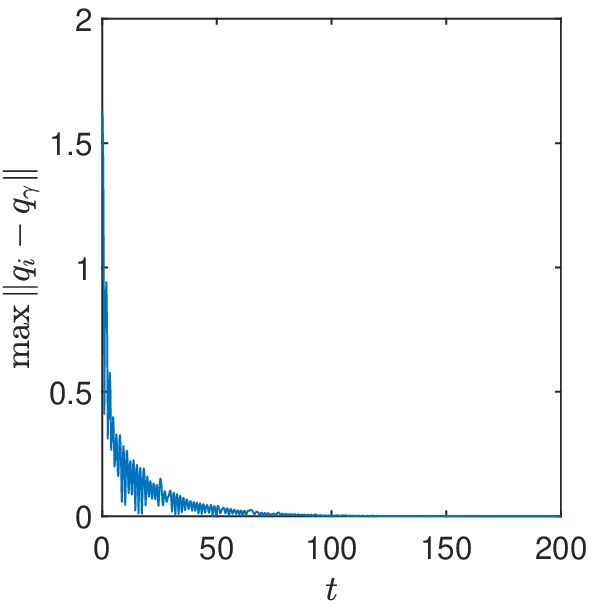}
\end{minipage}
\begin{minipage}{0.3\textwidth}
\centering
\includegraphics[width=\textwidth]{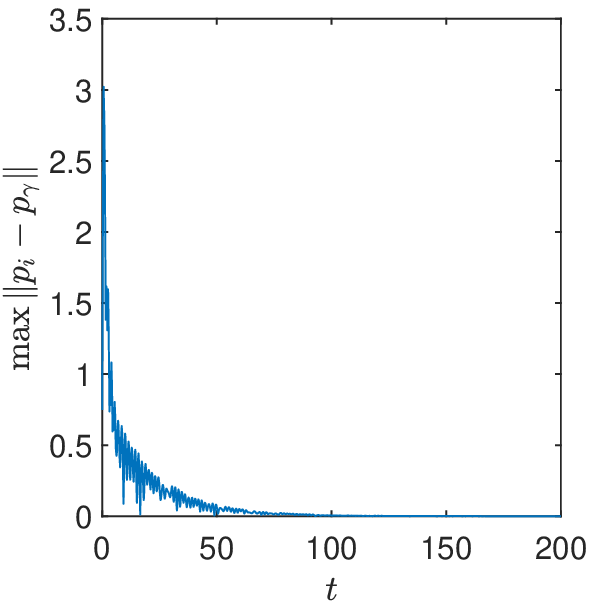}
\end{minipage}
\begin{minipage}{0.3\textwidth}
\centering
\includegraphics[width=\textwidth]{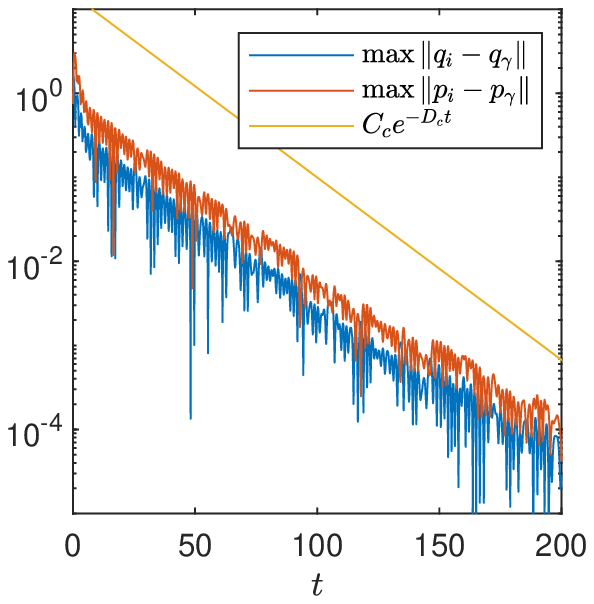}
\end{minipage}
\caption{The numerical results with flocking term and the same parameters with Figure \ref{fig2} }
\label{fig5}
\end{figure}
\begin{figure}[!ht]
\begin{minipage}{0.3\textwidth}
\centering
\includegraphics[width=\textwidth]{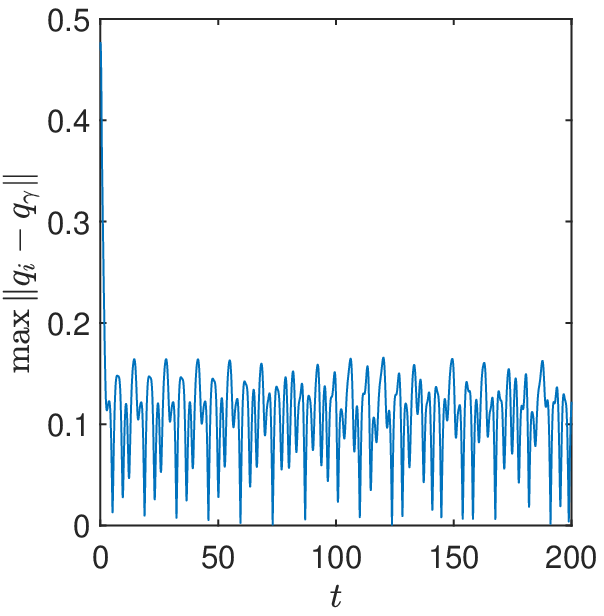}
\end{minipage}
\begin{minipage}{0.3\textwidth}
\centering
\includegraphics[width=\textwidth]{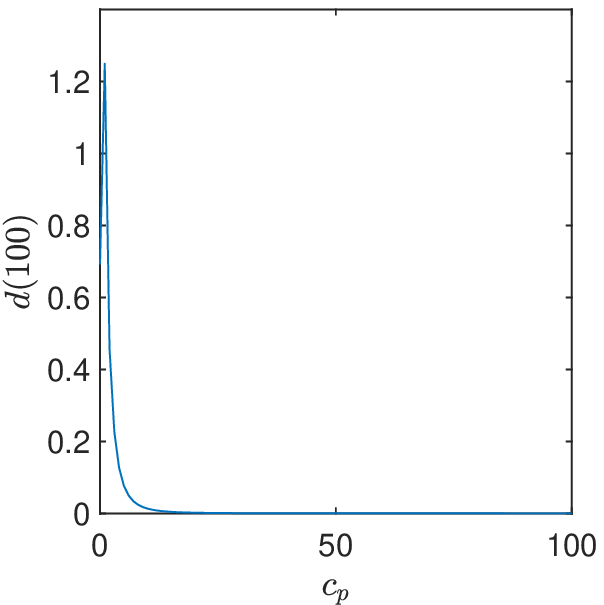}
\end{minipage}
\caption{The numerical results with flocking term and the same parameters with Figure \ref{fig4} }
\label{fig6}
\end{figure}
%위의 결과들을 flocking term 있을때 할것, 특히 logplot

Finally, we compare the target tracking problems on a sphere and flat space numerically. To compare the two cases, we impose the periodic boundary for the flat space and fix parameters such as $\sigma=1$, $c_q=5$, and $c_p=0.1$.
Let
\[u_\gamma=(a\cos t, a\sin t, a),\]
where $a=0.5$ and $u_i=u_\gamma$.
Then we can observe that the complete rendezvous occurs. See Figure \ref{fig7}.
If $u_i=0$, then we observe the practical rendezvous. See Figure \ref{fig8}.
%As we mentioned in section \ref{}, by the geometric property, we can see the difference in accordance with the manifolds.
%From the Figure \ref{fig6}, we can see that the target tracking is well established on a sphere because of the geometric property and the appropriate rotation operator.
%periodic boundary condition주고 flat space에서 결과 비교
\begin{figure}[!ht]
\begin{minipage}{0.23\textwidth}
\centering
\includegraphics[width=\textwidth]{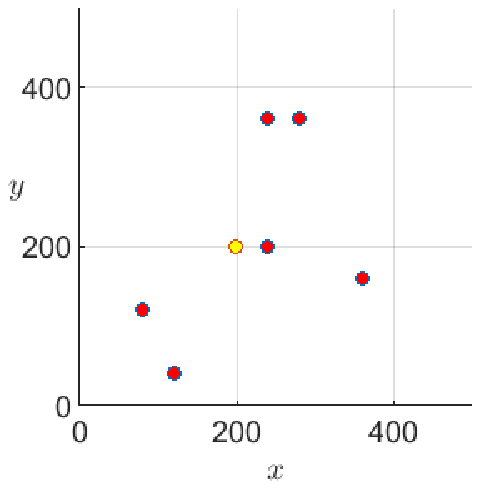}\\
(A)  $t=0$
\end{minipage}
\begin{minipage}{0.23\textwidth}
\centering
\includegraphics[width=\textwidth]{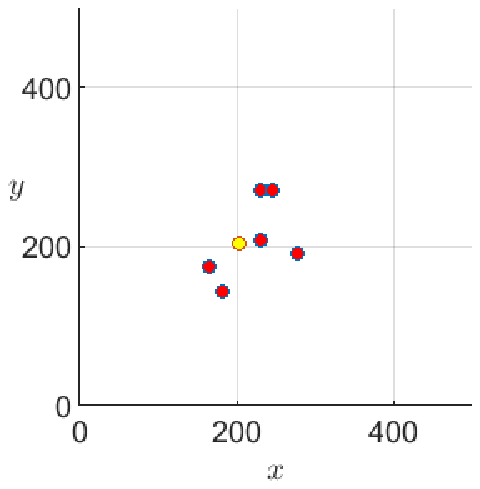}\\
(B) $t=5$
\end{minipage}
\begin{minipage}{0.23\textwidth}
\centering
\includegraphics[width=\textwidth]{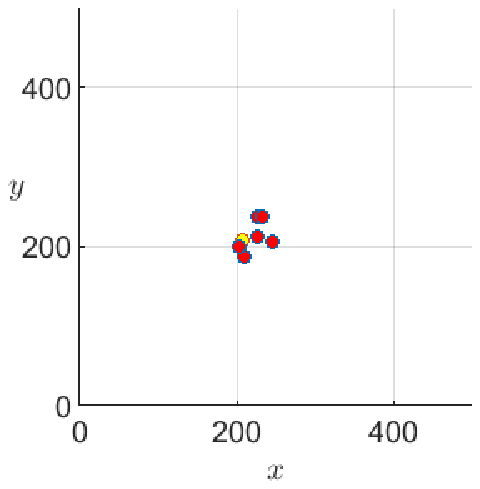}\\
(C) $t=10$
\end{minipage}
\begin{minipage}{0.23\textwidth}
\centering
\includegraphics[width=\textwidth]{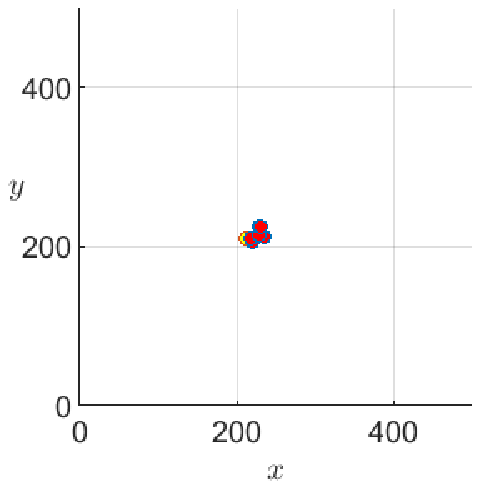}\\
(D) $t=15$
\end{minipage}
\begin{minipage}{0.23\textwidth}
\centering
\includegraphics[width=\textwidth]{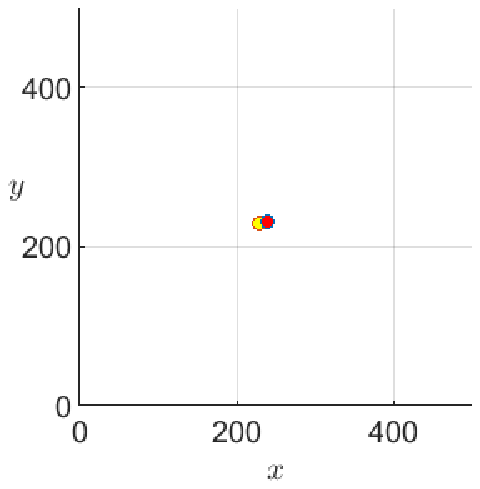}\\
(E) $t=40$
\end{minipage}
\begin{minipage}{0.23\textwidth}
\centering
\includegraphics[width=\textwidth]{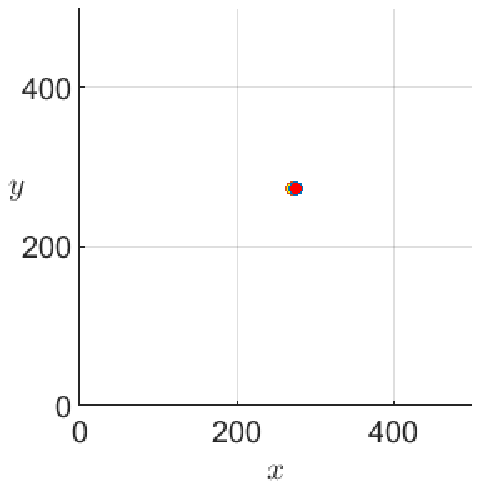}\\
(F) $t=100$
\end{minipage}
\begin{minipage}{0.23\textwidth}
\centering
\includegraphics[width=\textwidth]{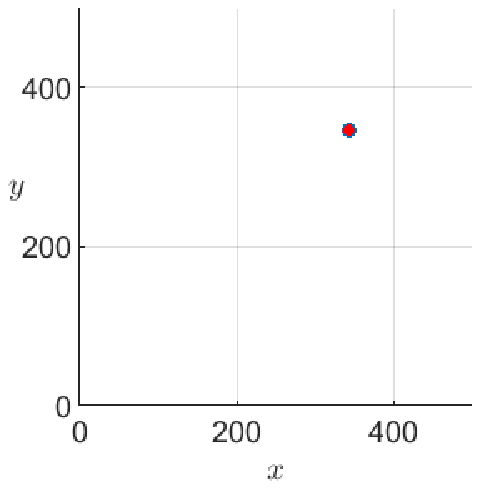}\\
(G) $t=200$
\end{minipage}
\begin{minipage}{0.23\textwidth}
\centering
\includegraphics[width=\textwidth]{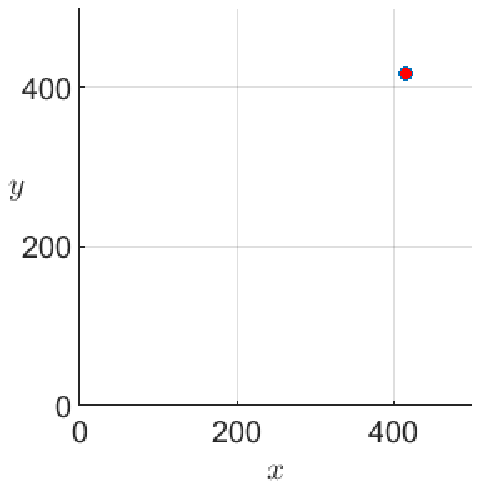}\\
(H) $t=300$
\end{minipage}
\caption{The snapshops of complete rendezvous on flat space}
\label{fig7}
\end{figure}
\begin{figure}[!ht]
\begin{minipage}{0.23\textwidth}
\centering
\includegraphics[width=\textwidth]{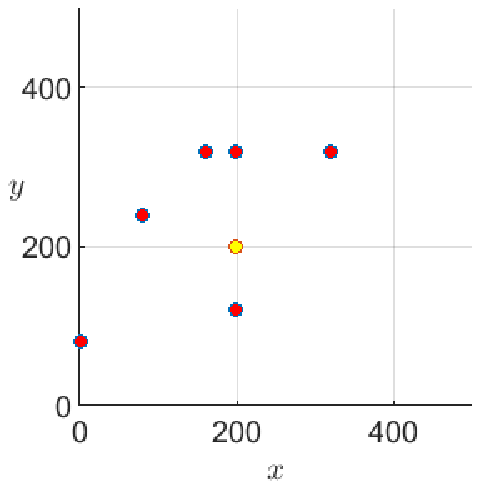}\\
(A)  $t=0$
\end{minipage}
\begin{minipage}{0.23\textwidth}
\centering
\includegraphics[width=\textwidth]{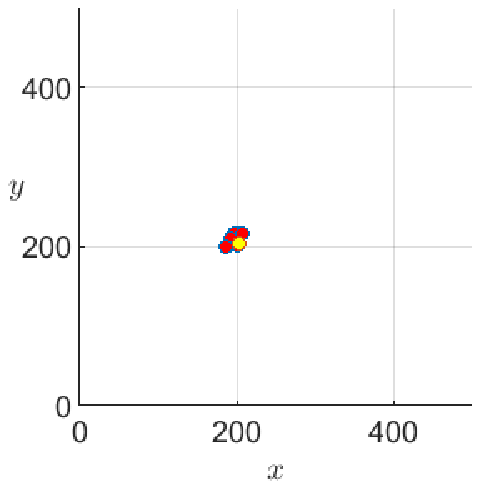}\\
(B) $t=5$
\end{minipage}
\begin{minipage}{0.23\textwidth}
\centering
\includegraphics[width=\textwidth]{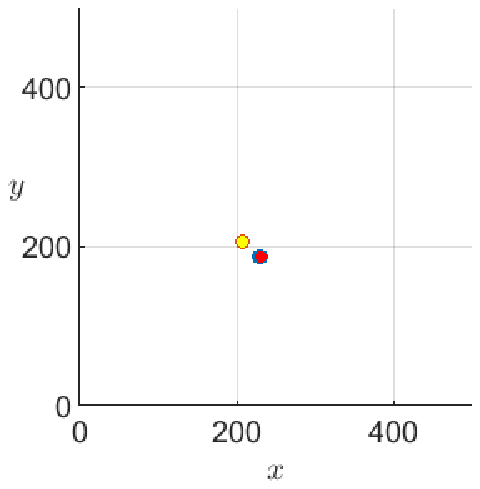}\\
(C) $t=10$
\end{minipage}
\begin{minipage}{0.23\textwidth}
\centering
\includegraphics[width=\textwidth]{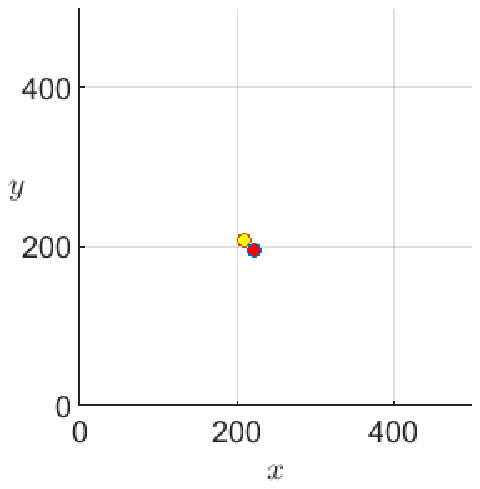}\\
(D) $t=15$
\end{minipage}
\begin{minipage}{0.23\textwidth}
\centering
\includegraphics[width=\textwidth]{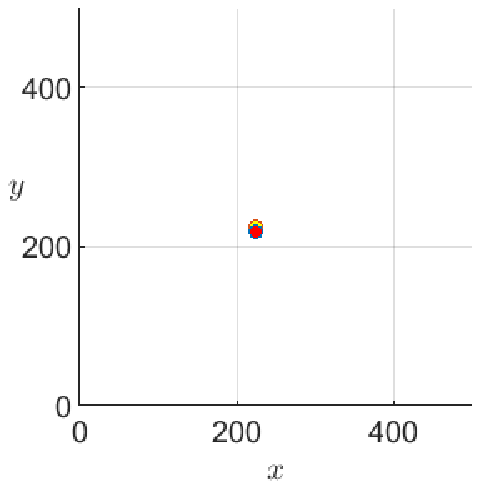}\\
(E) $t=40$
\end{minipage}
\begin{minipage}{0.23\textwidth}
\centering
\includegraphics[width=\textwidth]{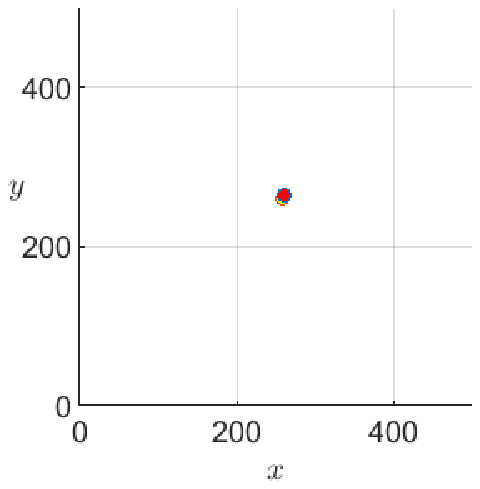}\\
(F) $t=100$
\end{minipage}
\begin{minipage}{0.23\textwidth}
\centering
\includegraphics[width=\textwidth]{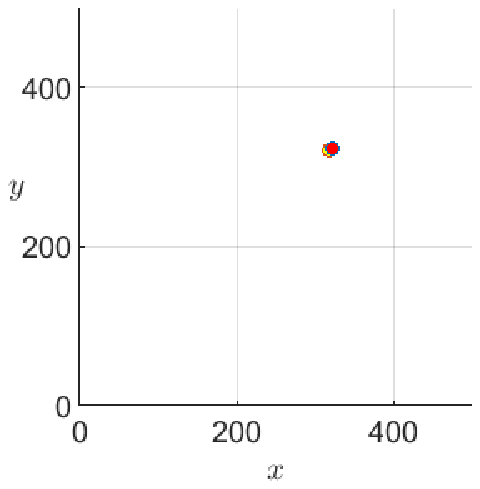}\\
(G) $t=200$
\end{minipage}
\begin{minipage}{0.23\textwidth}
\centering
\includegraphics[width=\textwidth]{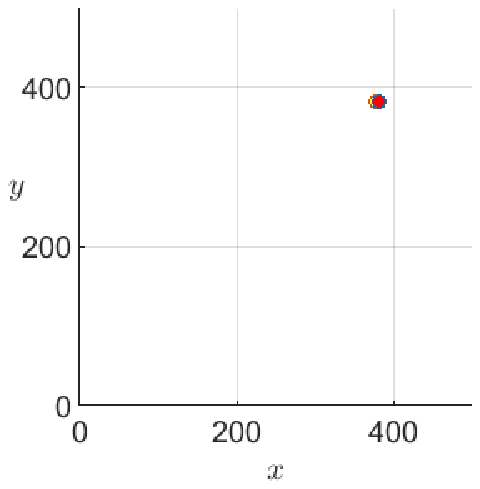}\\
(H) $t=300$
\end{minipage}
\caption{The snapshops of practical rendezvous on flat space}
\label{fig8}
\end{figure}

\section{Conclusion}\setcounter{equation}{0}\label{sec7}
In this paper, we proposed a novel model for target tracking on spherical geometry.  With the target's position, velocity, and acceleration, if the initial energy of agents is small or the bonding force between the target and each agent is larger than the one between agents, the complete rendezvous occurs. When only the information of position and velocity is known and the target's angular velocity and its time derivative are bounded, the practical rendezvous is obtained for relatively large intra-bonding forces.  The target tracking problems on $\mathbb{S}^2$ with time delay, white noises from the observation, and measurement are also interesting topics. These issues will be discussed in our future researches.

\appendix
\section{Properties of the admissible rotation operator}
In this part, we consider admissible rotation operators on a sphere and their properties.  The rotation operator appears naturally for defining the flocking on a sphere  \cite{C-K-S}.  Let $R_{\cdot\rightarrow \cdot}$ be Rodrigues' rotation operator given by
\[R_{x_k \rightarrow x_i} (v_k)=R(x_k,x_i)\cdot v_k\]
and for $x_k\ne x_i$,
\begin{align*}
\begin{aligned}
R(x_k,x_i):=
\langle x_k,x_i\rangle  I + x_i x_k^T - x_k x_i^T + (1-  \langle x_k,x_i\rangle) \left( \frac{x_k \times x_i}{|x_k \times x_i|} \right) \left( \frac{x_k \times x_i}{|x_k \times x_i|} \right)^T.
\end{aligned}
\end{align*}
Here, $x_k$, $x_i$ and $v_j$ are three dimensional column vectors. The rotation operator $\ro$ has many good properties we desired or needed to be physically established and we can construct a flocking model by replacing the velocity difference  term $v_i-v_j$ in the flat space to $ \rji v_j(t) - v_i(t)$. See \cite{C-K-S} for the details. However, there are some inconvenient points  due to the presence of singularity on $\ro$. Therefore, we can naturally ask whether such alternatives can be found.

The idea to find the alternative is as follows. First, classify the properties that the rotation operators must satisfy, and find all the operators that satisfy the properties. Next, we will choose one of those operators that meets our needs. Our option will be the simplest of the possible operators. This form has various advantages. It is convenient to calculate, and it shares most of the good properties of the rotation operator $\ro$ previously defined.  By removing the singularity, we easily show the global-in-time existence and uniqueness of the new model in \eqref{main1}. See  \cite{C-K-S} for the existence and uniqueness of the model with $\ro$.

 To construct a unit sphere model with the Newtonian equation, we need a modification of $v_j-v_i$ terms, which is the first motivation of the operators $\rji$ in \cite{C-K-S}. As we compute the velocity difference between $v_i$ and $v_j$ at the point $x_i$, we should transform $v_j$ into a tangential vector of the sphere at $x_i$. We note that  the typical ansatz for the flocking motion on a sphere is circle motions. In order to include circle motions along one great circle, the operator should coincide with a rotation operator in two dimensions, a $(x_i,x_j)$-plane. In other words, an admissible rotation operator $M$ from $z_1$ to $z_2$ can be a $3 \times 3$ matrix such that
\begin{subequations}
\label{eqn:adm}
\begin{align}
\label{eqn:adma}
&Mz_1 = z_2, \qquad M z_2 = 2\langle z_1, z_2\rangle  z_2 - z_1,\\
\label{eqn:admb}
&\langle Mv , z_2 \rangle = 0 \hbox{ for any } z_1, z_2 \in \D \hbox{ and } v \in T_{z_1} \D.
\end{align}
\end{subequations}
In the next proposition, we can prove that  the  admissible choices in \eqref{eqn:adm} for the rotation operator are equivalent to the following set.
\begin{align}
\label{eqn:ads}
\A_{z_1 \rightarrow z_2} := \left\{ \hR_{z_1 \rightarrow z_2} + a (z_1 \times z_2) (z_1 \times z_2)^T + b(z_1 - \langle z_1, z_2\rangle  z_2 ) (z_1 \times z_2)^T : a,b \in \R\right\},
\end{align}
where $\hR_{z_1 \rightarrow z_2}$ is the operator defined in \eqref{eqn:hrot}.

\begin{proposition}
%\label{lem:adm}
Suppose that unit vectors $z_1$ and $z_2$ are linearly independent. Then, a $3\times 3$ matrix $M$ satisfies \eqref{eqn:adm} if and only if $M \in \A_{z_1 \rightarrow z_2}$.
\end{proposition}

\begin{proof}
As two vectors $z_1$ and $z_2$ are perpendicular to $\zc$, operator $\hR_{z_1 \rightarrow z_2}$ satisfies \eqref{eqn:adm} from the direct computation. Note that $\langle z_1 \times z_2, z_i \rangle = 0$ for $i =1,2$. From this motivation, we naturally define
\begin{align}
\label{eqn:1adm}
M := \hR_{z_1 \rightarrow z_2} + a (z_1 \times z_2) (z_1 \times z_2)^T + b(z_1 - \langle z_1, z_2\rangle  z_2 ) (z_1 \times z_2)^T
\end{align}
for any $a, b \in \R$. Then $M$ satisfies \eqref{eqn:adma}. Also, as $z_2$ is perpendicular to both $z_1 \times z_2$ and $(z_1 - \langle z_1,  z_2\rangle  z_2 )$, we conclude \eqref{eqn:admb}.

Conversely, choose any $3\times 3$ matrix $M'$ satisfying \eqref{eqn:adm}. As $z_1$ and $z_2$ are linearly independent, $\{z_2,~ z_1 - \langle z_1, z_2\rangle  z_2, ~\zc\}$ are a basis of $\R^3$. Therefore, there are $a, b,c  \in \R$ such that
\begin{align}
\label{eqn:adm21}
M' \frac{\zc}{\| \zc \|^2} = a  (z_1 \times z_2) + b (z_1 - \langle z_1,  z_2\rangle  z_2 ) + c z_2.
\end{align}
From \eqref{eqn:admb} and $\zc \in T_{z_1} \D$, it follows that $c=0$. Therefore, we conclude that
\begin{align*}
M \zc = M' \zc
\end{align*}
for $M$ given in \eqref{eqn:1adm}. On the other hand, \eqref{eqn:adma} show that
\begin{align}
\label{eqn:adm22}
M(z_2) = M'(z_2)\quad \hbox{ and }\quad M (z_1 - \langle z_1, z_2\rangle  z_2) = M' (z_1 - \langle z_1, z_2\rangle  z_2).
\end{align}
From \eqref{eqn:adm21} and \eqref{eqn:adm22}, we obtain that $M = M'$.
\end{proof}

The set $\A_{z_1 \rightarrow z_2}$  includes the rotation operators $\rot$ and $\hR_{z_1 \rightarrow z_2}$ given in \cite{C-K-S} and \eqref{eqn:hrot}, respectively. Here, if we take the following values in \eqref{eqn:1adm}:
\[a = \frac{1 - \langle z_1, z_2\rangle }{\|\zc\|^2}\quad \mbox{and} \quad b=0,\]
then the matrix coincides with $\rot$, which preserves the modulus of each vectors. See Lemma~2.3 in \cite{C-K-S}. Among several choices in the admissible set in \eqref{eqn:ads}, $\hR_{z_1 \rightarrow z_2}$ can be regarded as the simplest choice such that $a=b=0$ in \eqref{eqn:ads}. Moreover, there is no singularity compared to the previous rotation operator $\ro$. In addition to this simplicity, the rotation operator $\hR_{z_1 \rightarrow z_2}$ also share the following desired transport properties.

\begin{lemma}
\label{lem:hrot}
For $z_1, z_2 \in \D$,  $\hR_{z_1 \rightarrow z_2}$ given in \eqref{eqn:hrot} satisfies \eqref{eqn:adm}. Furthermore, we have
\begin{align}
\label{eqn:2hrot}
\hR_{z_1 \rightarrow z_2}^T = \hR_{z_2\rightarrow z_1}
\end{align}
and
\begin{align*}
%\label{eqn:3hrot}
\hR_{z_1 \rightarrow z_2}^T \hR_{z_1 \rightarrow z_2} (z_1 ) = z_1, \quad  \hR_{z_1 \rightarrow z_2}^T \hR_{z_1 \rightarrow z_2} (z_2) = z_2.
\end{align*}
\end{lemma}

\begin{proof}
As two vectors $z_1$ and $z_2$ are perpendicular to $\zc$, the properties in  \eqref{eqn:adm} follow from the direct computation. Also, since the transpose is the linear operator, we have
\begin{align*}
\hR_{z_1 \rightarrow z_2}^T &= \langle z_1, z_2\rangle  I - z_2 z_1^T + z_1 z_2^T,
\end{align*}
and we conclude \eqref{eqn:2hrot}. From \eqref{eqn:adm} and \eqref{eqn:2hrot}, it holds that
\begin{align*}
%\label{eqn:rot21}
\hR_{z_1 \rightarrow z_2}^T \hR_{z_1 \rightarrow z_2} (z_1 ) = \hR_{z_1 \rightarrow z_2}^T (z_2) = z_1
\end{align*}
and
\begin{align*}
%\label{eqn:rot22}
\hR_{z_1 \rightarrow z_2}^T \hR_{z_1 \rightarrow z_2} (z_2) = \hR_{z_1 \rightarrow z_2}^T(2  \langle z_1, z_2\rangle  z_2 - z_1) = 2\langle z_1, z_2\rangle  z_1 - (2\langle z_1, z_2\rangle z_1 - z_2) = z_2.
\end{align*}
\end{proof}

While the two operators $\rot$ and $\hR_{z_1 \rightarrow z_2}$ coincide on the $(z_1,z_2)$-plane from Lemma~\ref{lem:hrot}, the following lemma gives us one difference between the two operators. We can show that $\hro$ gives a map between two tangent spaces although the operator is not a bijection if $\langle z_1, z_2\rangle  = 0$.

\begin{lemma}
%\label{prop:htan}
$\hR_{z_1 \rightarrow z_2} |_{T_{z_1} \D}$ is a map from $T_{z_1} \D$ to $T_{z_2} \D$. Furthermore, if $\langle z_1, z_2\rangle  \neq 0$, then $\hR_{z_1 \rightarrow z_2} |_{T_{z_1} \D}$ is a bijection from $T_{z_1} \D$ to $T_{z_2} \D$.
\end{lemma}
\begin{proof}
As $\D$ is a unit sphere, $v \in T_{y} \D$~ if and only if~ $\langle v, y\rangle  = 0$ for any $y \in \R^3$.
Thus, we have
\begin{align}
\label{eqn:htan12}
\langle v, z_1\rangle  = 0\quad \hbox{ for any vector } v \in T_{z_1} \D.
\end{align}
From \eqref{eqn:adma} and \eqref{eqn:2hrot}, it holds that for any $v \in \R^3$,
\begin{align}
\label{eqn:htan11}
(\hR_{z_1 \rightarrow z_2} v) \cdot z_2 = v^T \hR_{z_1 \rightarrow z_2}^T z_2 = v^T \hR_{z_1 \rightarrow z_2}r z_2 = v^T z_1=\langle v, z_1\rangle.
\end{align}
By \eqref{eqn:htan12} and \eqref{eqn:htan11}, we conclude that
\begin{align*}
%\label{eqn:htan13}
(\hR_{z_1 \rightarrow z_2} v) \cdot z_2 = 0 \hbox{ and thus } \hR_{z_1 \rightarrow z_2} v \in T_{z_2} \D \hbox{ for any vector } v \in T_{z_1} \D.
\end{align*}

We now assume that $\langle z_1, z_2\rangle  \neq 0$ and show that $\hR_{z_1 \rightarrow z_2} |_{T_{z_1} \D}$ is bijective between two tangent spaces. First, if $z_1 = z_2$ or $z_1 = -z_2$, we get $\hR_{z_1 \rightarrow z_2} = I$ and $\hR_{z_1 \rightarrow z_2} = -I$. If not, $z_1$ and $z_2$ are linearly independent. From the assumption,
$\hR_{z_1 \rightarrow z_2}(z_1 \times z_2) = \langle z_1, z_2\rangle  (z_1 \times z_2)$ is a nonzero vector. Combining this with \eqref{eqn:adma}, we conclude that $\hR_{z_1 \rightarrow z_2} |_{T_{z_1} \D}$ is surjective in $T_{z_2} \D$ and thus the determinant of $\hR_{z_1 \rightarrow z_2}$ is nonzero. As the inverse function of $\hR_{z_1 \rightarrow z_2}$ exists, we conclude that this lemma holds.
\end{proof}

\bibliographystyle{amsplain}

\begin{thebibliography}{99}

\bibitem{B-C-B-C}Bak, S., Chau, D. P., Badie, J., Corvee, E., Brémond, F., and Thonnat, M. (2012, September). Multi-target tracking by discriminative analysis on Riemannian manifold. In 2012 19th IEEE International Conference on Image Processing (pp. 1605-1608). IEEE.

\bibitem{B2} Blackman, S. S. (2004). Multiple hypothesis tracking for multiple target tracking. IEEE Aerospace and Electronic Systems Magazine, 19(1), 5-18.

\bibitem{B1} Blackman, S. S. (1986). Multiple-target tracking with radar applications. Dedham.


\bibitem{C-C-H}Chi, D., Choi, S.-H. and Ha,  S.-Y.(2014). Emergent behaviors of a holonomic particle system on a sphere. Journal of Mathematical Physics,  55, 052703.

\bibitem{C-C-H2}Choi, S.-H., Cho, J. and Ha, S.-Y.(2016).  Practical quantum synchronization for the Schrödinger–Lohe system. Journal of Physics A: Mathematical and Theoretical,  49(20), 205203.

\bibitem{C-K-S}  Choi, S.-H.,  Kwon, D. and Seo, H. (2020). Cucker-Smale type flocking models on a sphere. arXiv preprint arXiv:2010.10693.

\bibitem{C-K-S1} Choi, S.-H.,  Kwon, D. and Seo, H.: Uniform position alignment estimate of a spherical flocking model with inter-particle bonding forces. arXiv preprint arXiv:2101.00791.

\bibitem{C-K-S2}  Choi, S.-H.,  Kwon, D. and Seo, H.: Flocking formation and stabilizer of boosted cooperative control on a sphere, preprint.


\bibitem{D-B}Daeipour, E., and Bar-Shalom, Y. (1995). An interacting multiple model approach for target tracking with glint noise. IEEE Transactions on Aerospace and Electronic Systems, 31(2), 706-715.

\bibitem{D-S-A-Y}Deghat, M., Shames, I., Anderson, B. D., and Yu, C. (2014). Localization and circumnavigation of a slowly moving target using bearing measurements. IEEE Transactions on Automatic Control, 59(8), 2182-2188.
\bibitem{Olfati} Olfati-Saber, R. (2006). Flocking for multi-agent dynamic systems: Algorithms and theory. IEEE Transactions on automatic control, 51(3), 401-420.


\bibitem{H}  He, T.,  Vicaire, P.,  Yan, T.,  Luo, L.,  Gu, L.,  Zhou, G.,
 Stoleru, R.,  Cao, Q.,  Stankovic,J. A. and  Abdelzaher, T.(2006). Achieving real-time target tracking usingwireless sensor networks. 12th IEEE Real-Time and Embedded Technology and Applications Symposium (RTAS'06). IEEE, 2006.

\bibitem{H-H}Hu, J., and Hu, X. (2010). Nonlinear filtering in target tracking using cooperative mobile sensors. Automatica, 46(12), 2041-2046.



\bibitem{L-Z-C-Z-Z} Jia-qiang, L., Rong-hua, Z., Jin-li, C., Chun-yan, Z., and Yan-ping, Z. (2016). Target tracking algorithm based on adaptive strong tracking particle filter. IET Science, Measurement \& Technology, 10(7), 704-710.


\bibitem{L-J} Li, X. R., and Jilkov, V. P. (2004, August). A survey of maneuvering target tracking: approximation techniques for nonlinear filtering. In Signal and Data Processing of Small Targets 2004 (Vol. 5428, pp. 537-550). International Society for Optics and Photonics.

\bibitem{M-C}Madyastha, V. K., and Caliset, A. J. (2005, June). An adaptive filtering approach to target tracking. In Proceedings of the 2005, American Control Conference, 2005. (pp. 1269-1274). IEEE.

\bibitem{O-S-S} Oh, S., Sastry, S., and Schenato, L. (2005, April). A hierarchical multiple-target tracking algorithm for sensor networks. In Proceedings of the 2005 IEEE International Conference on Robotics and Automation (pp. 2197-2202). IEEE.

\bibitem{S-B}Semnani, S. H., and Basir, O. A. (2014). Semi-flocking algorithm for motion control of mobile sensors in large-scale surveillance systems. IEEE transactions on cybernetics, 45(1), 129-137.


\bibitem{S-D-A} Shames, I., Dasgupta, S., Fidan, B., and Anderson, B. D. (2011). Circumnavigation using distance measurements under slow drift. IEEE Transactions on Automatic Control, 57(4), 889-903.

\bibitem{S-S-D} Sworder, D. D., Singer, P. F., Doria, D., and Hutchins, R. G. (1993). Image-enhanced estimation methods. Proceedings of the IEEE, 81(6), 797-814.



\bibitem{Tes12} Teschl, G.(2012).  Ordinary differential equations and dynamical systems. American Mathematical Soc. 140.

\bibitem{X-Z-S}Xu, Enyang, Zhi Ding, and Soura Dasgupta. (2011). Target tracking and mobile sensor navigation in wireless sensor networks. IEEE Transactions on mobile computing 12(1), 177-186.

\bibitem{Y-S-C}Yang, Z., Shi, X., and Chen, J. (2013). Optimal coordination of mobile sensors for target tracking under additive and multiplicative noises. IEEE Transactions on Industrial Electronics, 61(7), 3459-3468.
\bibitem{Y-Y-J}Yin, Guisheng, Yanbo Li, and Jing Zhang. (2008). The Research of Video Tracking System Based on Virtual Reality. International Conference on Internet Computing in Science and Engineering. IEEE.

%
%\bibitem{Alt}Alt, W.: Biased random walk models for chemotaxis and related diffusion approximations. Journal of mathematical biology 9 (1980), 147-177.
%
%
%\bibitem{B-M}Bishop, A. N., and Basiri, M.: Bearing-only triangular formation control on the plane and the sphere. 18th Mediterranean Conference on Control and Automation, MED'10, IEEE (2010),  790-795.
%
%
%\bibitem{C-F-R-T}  Carrillo, J. A., Fornasier, M., Rosado, J. and Toscani, G.: Asymptotic flocking dynamics for the kinetic Cucker-Smale model. SIAM Journal on Mathematical Analysis 42 (2010), 218-236.
%
%
%\bibitem{C-K-S} Choi, S.-H.,  Kwon, D. and Seo, H.: Cucker-Smale type flocking models on a sphere. arXiv preprint (2020), arXiv:2010.10693. %\url{https://arxiv.org/abs/2010.10693}
%
%
%\bibitem{C-K-S2} Choi, S.-H.,  Kwon, D. and Seo, H.: Uniform position alignment estimate of spherical flocking model with inter-particle bonding forces. arXiv preprint  (2021), arXiv:2101.00791. %\url{https://arxiv.org/abs/2101.00791}
%
%%Choi, S.-H.,  Kwon, D. and   Seo, H.: Cucker-Smale type flocking models on a sphere, preprint.
%
%
%
% \bibitem{C-S2} Cucker, F. and Smale, S.:  Emergent behavior in flocks. IEEE Transactions on Automatic Control
% 52 (2007), 852-862.
%
%\bibitem{D-M1}  Degond, P. and Motsch, S.: Large-scale dynamics of the Persistent Turing Walker model of fish behavior. Journal of Statistical Physics  131 (2008), 989-1022.
%
%
%
%\bibitem{F-E}Fetecau, R. C. and Eftimie, R.: An investigation of a nonlocal hyperbolic model for selforganization of biological groups.Journal of Mathematical Biology  61 (2010), 545-579.
%
%
%
%\bibitem{G-C-F} Ganganath, N., Yuan, W., Cheng, C. T., Fernando, T., and Iu, H. H.: Territorial marking for improved area coverage in anti-flocking-controlled mobile sensor networks. In 2018 IEEE International Symposium on Circuits and Systems (ISCAS), IEEE (2018), 1-4.
%
%
%\bibitem{H-C}Huth, A. and  Wissel, C.: The simulation of the movement of fish schools. Journal of theoretical biology 156 (1992), 365-385.
%
%\bibitem{K-G}Khalil, Hassan K., and Grizzle, Jessy W.: Nonlinear systems. Vol. 3. Upper Saddle River, NJ: Prentice hall, 2002.
%
%\bibitem{K-R}Kim, A., and Ryan M. E.: Active visual SLAM for robotic area coverage: Theory and experiment. The International Journal of Robotics Research 34 (2015), 457-475.
%
%
%\bibitem{Kuromoto} Kuramoto, Y.: International symposium on mathematical problems in theoretical
%  physics.   Lecture notes in Physics, 1975.
%
%
%
%
%\bibitem{La-S} Lageman, C. and Sun, Z.:  Consensus on spheres: Convergence analysis and perturbation theory. In 2016 IEEE 55th Conference on Decision and Control (CDC), IEEE (2016),  19-24.
%
%
%
%
%\bibitem{L2} Li, W.: Collective motion of swarming agents evolving on a sphere manifold: A fundamental framework and characterization. Scientific reports 5 (2015), 13603.
%
%
%
%\bibitem{L-C} Li, W. and Chen, G.: The designated convergence rate problem of consensus or flocking of double-integrator agents with general non-equal velocity and position couplings. IEEE Transactions on Automatic Control 62  (2016), 412-418.
%
%
%
%\bibitem{L-S} Li, W. and Spong, M. W.: Unified cooperative control of multiple agents on a sphere for different spherical patterns. IEEE Transactions on Automatic Control 59 (2013), 1283-1289.
%
%
%
%\bibitem{M-G} Markdahl, J. and Goncalves, J.: Global converegence properties of a consensus protocol on the n-sphere. In 2016 IEEE 55th Conference on Decision and Control (CDC), IEEE (2016), 3487-3492.
%
%
%\bibitem{M-T-G} Markdahl, J., Thunberg, J., and Goncalves, J.: High-dimensional Kuramoto models on Stiefel manifolds synchronize complex networks almost globally. Automatica 113  (2020), 108736.
%
%
%
%
%\bibitem{M-T-G2} Markdahl, J., Thunberg, J. and Goncalves, J.: Almost global consensus on the $n$-sphere. IEEE Transactions on Automatic Control 63 (2017), 1664-1675.
%
%
%
%\bibitem{M} Murray, R. M.: Recent research in cooperative control of multivehicle systems. Journal of Dynamic Systems, Measurement, and Control  129 (2007), 571-583.
%
%
%
%
%
%
%
%\bibitem{O} Olfati-Saber, R.: Flocking for multi-agent dynamic systems: Algorithms and theory. IEEE
%Transactions on automatic control 51  (2006), 401-420.
%
%
%
%
%
%\bibitem{PKH10}Park, J., Kim, H. J.  and Ha, S.-Y.:  Cucker-Smale flocking with inter-particle bonding forces. IEEE Transactions on Automatic Control  55 (2010), 2617-2623.
%
%
%\bibitem{P-D}Pereira, P. O. and Dimarogonas, D. V.: Family of controllers for attitude synchronization on the sphere. Automatica  75  (2017), 271-281.
%
%
%\bibitem{P-P} Proskurnikov, A. V. and Parsegov, S. E.: Problem of uniform deployment on a line segment for second-order agents. Automation and Remote Control 77 (2016), 1248-1258.
%
%
%
%    \bibitem{R}Ross, H.: Fly around the world with a solar powered airplane. In: The 26th Congress of ICAS and 8th AIAA ATIO, (2008),  8954.
%
%
%
%    \bibitem{S-W-X}Sai, L., Wei, Z., Xueren, W. The development status and key technologies of solar powered unmanned air vehicle.  IOP conference series: materials science and engineering,  IOP Publishing (2017), 012011.
%
%
%\bibitem{S-H}Schienbein, M. and Gruler, H.: Langevin equation, Fokker-Planck equation and cell migration. Bulletin of Mathematical Biology 55 (1993), 585-608.
%
%
%
%
%\bibitem{S-S-Z-T}Skobelev, P. O., Simonova, E. V., Zhilyaev, A. A., and Travin, V. S.: Application of multi-agent technology in the scheduling system of swarm of earth remote sensing satellites. Procedia Computer Science 103  (2017), 396-402.
%
%\bibitem{S-M-Z-H-H} Song, W., Markdahl, J., Zhang, S., Hu, X., and Hong, Y.: Intrinsic reduced attitude formation with ring inter-agent graph. Automatica 85 (2017), 193-201.
%
%\bibitem{S} Sun, Z.: Cooperative coordination and formation control for multi-agent systems. Springer, 2018.
%
%
%\bibitem{T-M-B-G} Thunberg, J., Markdahl, J., Bernard, F., and Goncalves, J.: A lifting method for analyzing distributed synchronization on the unit sphere. Automatica 96  (2018), 253-258.
%
%
%
%
%\bibitem{T-T} Toner, J. and Tu, Y.: em Flocks, herds, and schools: a quantitative theory of flocking. Physical Review E  58 (1998), 4828.
%
%
%
%\bibitem{T-B} Topaz, C. M. and Bertozzi, A. L.: Swarming patterns in a two-dimensional kinematic model for biological groups. SIAM Journal on Applied Mathematics  65 (2004), 152-174.
%
%
%
%
%
%
%\bibitem{V-C-B-C-S} Vicsek, T., Czir\'{o}k, Ben-Jacob, E., Cohen, I. and Schochet, O.:  Novel type of phase transition in a system of self-driven particles. Physical Review Letters  75 (1995), 1226-1229.
%
%
%
%\bibitem{Wi} Winfree, A. T.: Biological rhythms and the behavior of populations of coupled oscillators. Journal of Theoretical Biology 16 (1967), 15-42.
%
%
%\bibitem{Z-X-H-M-T} Zhu, B., Xie, L., Han, D., Meng, X., and Teo, R.: A survey on recent progress in control of swarm systems. Science China Information Sciences 60 (2017), 070201.





\end{thebibliography}

\end{document}